\documentclass{article}

\usepackage{hyperref}
\usepackage{verbatim}
\usepackage{amsmath}
\usepackage{amssymb}
\usepackage{amsthm}
\usepackage{comment}
\usepackage[backend=bibtex,style=alphabetic,citestyle=alphabetic]{biblatex}
\usepackage{mathtools}
\usepackage{bbm}
\usepackage[utf8]{inputenc}
\usepackage{tikz-cd}
\usepackage{wasysym}
\usepackage{varwidth}
\usepackage{mathtools}
\usepackage{xspace}
\usepackage{aliascnt}
\usepackage{manfnt}
\usepackage{mathrsfs}
\usepackage[bbgreekl]{mathbbol}
\usepackage[title,toc]{appendix}

\usepackage{titlesec}
\titleformat{\paragraph}[runin]{\normalfont\normalsize\itshape}{\theparagraph}{}{}[.] 
\titlespacing{\paragraph}{0pt}{0pt}{*1} 

\usepackage[2cell,ps]{xy}
\input xy
\UseAllTwocells
\xyoption{all}

\mathchardef\mhyphen="2D
\DeclareMathOperator{\Hom}{Hom}

\DeclareMathOperator{\pt}{pt}

\DeclareMathOperator{\Ob}{Ob}
\DeclareMathOperator{\Exact}{Exact}

\DeclareMathOperator{\Id}{Id}

\newcommand{\Spaces}{\mathscr{S}}

\DeclareMathOperator{\sset}{sSet}

\DeclareMathOperator{\Corr}{Corr}
\newcommand{\cInt}[2]{#2_#1}

\newcommand{\TTT}{\mathcal{T}}

\newcommand{\CORR}{\operatorname{CORR}}
\newcommand{\SPAN}{\operatorname{SPAN}}

\newcommand{\squCorns}[4]{
\begin{tikzpicture}[anchor=base, baseline,inner sep=0]%
\node[scale=0.7] (a) at (0,0){%
\begin{tikzcd}[ampersand replacement=\&,row sep=tiny,column sep=tiny]%
#1 \ar{r} \ar{d} \& #2\ar{d}\\
#3 \ar{r}\& #4
\end{tikzcd}%
};%
\end{tikzpicture}
}

\DeclareMathOperator{\Spc}{Spc}
\newcommand{\CAT}[1]{#1\text{-}\operatorname{Cat}}

\newcommand{\HaugsengU}{U_{\text{seg}}}
\newcommand{\Maps}{\mathcal{M}aps}

\newcommand{\pbsquare}{\square}

\newcommand{\pbcorner}{\lrcorner}

\newcommand{\pbcube}{\text{\mancube}}
\newcommand{\pbcubecorner}{\lrcorner_c}

\newcommand{\CCC}{\mathcal{C}}

\newcommand{\AinfinityNU}{\widetilde{\mathbbm{A}_\infty}}
\newcommand{\Ainfinity}{\mathbbm{A}_\infty}

\newcommand{\aug}[1]{A(#1)}
\newcommand{\grid}[1]{\operatorname{grid}}

\newcommand{\OrdSet}{\mathbb{\Delta}}

\newcommand{\AugOrdSet}{\mathbb{\Delta}_+}

\DeclareMathOperator{\point}{\pt}

\DeclareMathOperator{\AlgOper}{Alg}
\newcommand{\cubeInfinity}[1]{\square_\infty^{#1}}

\newcommand{\CorrCube}[1]{\square_{\Corr}^{#1}}
\newcommand{\CorrCubeSegal}[1]{\mathcal{I}(\square_{\Corr}^{#1})}
\newcommand{\MapsSegal}{\Maps_{\mathcal{I}}}

\newcommand{\sgn}[1]{\text{sgn}(#1)}
\newcommand{\DDD}{\mathcal{D}}

\newcommand{\ord}[1]{\left< #1 \right> }
\newcommand{\upperSeg}[1]{\mathcal{U}[#1]}
\newcommand{\lowerSeg}[1]{\mathcal{L}[#1]}
\newcommand{\ordop}[1]{\Delta_{#1} }

\newcommand{\tik}{\begin{tikzcd}}
\newcommand{\tak}{\end{tikzcd}}

\newcommand{\coprOver}[1]{\sqcup_{{}_{#1}}}
\newcommand{\prodOver}[1]{\times_{{}_{#1}}}
\newcommand{\RR}{\mathbbm{R}}

\makeatletter
\def\latearrow#1#2#3#4{%
  \toks@\expandafter{\tikzcd@savedpaths\path[/tikz/commutative diagrams/every arrow,#1]}%
  \global\edef\tikzcd@savedpaths{%
    \the\toks@%
    (\tikzmatrixname-#2)
    to%
    node[/tikz/commutative diagrams/every label] {$#4$}
    (\tikzmatrixname-#3)
;}}
\makeatother

\def\stik#1#2{
\begin{tikzpicture}[baseline= (a).base]%
\node[scale=#1] (a) at (0,0){%
\begin{tikzcd}[ampersand replacement=\&]%
#2
\end{tikzcd}%
};%
\end{tikzpicture}
}

\def\nnstik#1#2{
\begin{tikzpicture}[baseline= (a).base]%
\node[scale=#1] (a) at (0,0){%
\begin{tikzcd}[ampersand replacement=\&,row sep=small,column sep=small]%
#2
\end{tikzcd}%
};%
\end{tikzpicture}
}

\newtheorem{Theorem}{Theorem}[section]
\newtheorem*{unnu-theorem}{Theorem}

\newtheorem{Proposition}{Proposition}[section]

\newaliascnt{Corollary}{Proposition}
\newtheorem{Corollary}[Corollary]{Corollary}
\aliascntresetthe{Corollary}

\newaliascnt{Lemma}{Proposition}
\newtheorem{Lemma}[Lemma]{Lemma}
\aliascntresetthe{Lemma}

\newaliascnt{Fact}{Proposition}
\newtheorem{Fact}[Fact]{Fact}
\aliascntresetthe{Fact}

\newtheorem* {Claim}{Claim}

\newtheoremstyle{break}
  {}
  {}
  {\itshape}
  {}
  {\bfseries}
  {.}
  {\newline}
  {}
\theoremstyle{break}

\theoremstyle{definition}
\newaliascnt{Definition}{Proposition}
\newtheorem{Definition}[Definition]{Definition}
\aliascntresetthe{Definition}

\newaliascnt{Construction}{Proposition}
\newtheorem{Construction}[Construction]{Construction}
\aliascntresetthe{Construction}

\theoremstyle{remark}
\newaliascnt{Notation}{Proposition}
\newtheorem{Notation}[Notation]{Notation}
\aliascntresetthe{Notation}

\newaliascnt{Remark}{Proposition}
\newtheorem{Remark}[Remark]{Remark}
\aliascntresetthe{Remark}

\newaliascnt{Example}{Proposition}
\newtheorem{Example}[Example]{Example}
\aliascntresetthe{Example}

\def\sectionautorefname~#1\null{%
\S#1\null
}

\def\equationautorefname~#1\null{%
(#1)\null
}

\begin{document}

\title{Higher Segal spaces and Lax $\Ainfinity$-algebras}
\author{Adam Gal, Elena Gal}
\maketitle

\begin{abstract}
The notion of a higher Segal space was introduced by Dyckerhoff and Kapranov in \cite{KapranovDyckerhoff} as a general framework for studying higher associativity inherent in a wide range of mathematical objects.
In the present work we formalize the connection between this notion and the notion of $\Ainfinity$-algebra. We introduce the notion of a "$d$-lax $\Ainfinity$-algebra object” which generalizes the notion of an $\Ainfinity$-algebra object in a precise sense.
We describe a construction that assigns to a simplicial object $S_\bullet$ in a category $\Spaces$ a datum of higher associators. We show that this datum defines a $d$-lax $\Ainfinity$-algebra object in the category of correspondences in $\Spaces$ precisely when $S_\bullet$ is a $(d+1)$-Segal object. 
More concretely we prove that for $n\geq d$ the  "$n$-dimensional associator" is invertible.
The so called "upper" and "lower" $d$-Segal conditions which originally come from the geometry of polytopes appear naturally in our construction as the two conditions which together imply the invertibility of the $d$-dimensional associator.
A corollary is that for $d=2$, our construction defines an $\Ainfinity$-algebra in the $(\infty,1)$-category of correspondences in $\Spaces$ with the $2$-Segal conditions implying invertibility of all associativity data.    
\end{abstract}

\tableofcontents
\section{Introduction}
In \cite{KapranovDyckerhoff} the authors introduce the concept of a $d$-Segal object and extensively study the $d=2$ case as a framework which generalizes many examples of associative algebras of categorical origin appearing in the literature.
One of the main classes of examples are Hall algebras, which include quantum groups.
The origin of the 2-Segal objects which give rise to these algebras is the Waldhausen construction and generalizations of it.

A generalization of Waldhausen construction studied in \cite{PoguntkeSegal} and \cite{DyckerhoffJasso} leads to a class of examples of $d$-Segal spaces for $d>2$. The result of the current work is that these also give rise to (suitably generalized) associative algebras.

As is already evident in very early works on the subject, all constructions of associative algebras arising in this way factor through a construction of \emph{correspondences} or \emph{spans}. This intermediate step will be the focus of this article. To then continue on to produce more familiar objects one needs to move from correspondences to an algebraic setting, such as vector spaces, linear categories, etc. This procedure is what is called a \emph{theory with transfer} in \cite{KapranovDyckerhoff} and several examples are given in loc. cit. Another important, and more involved, example is the construction of the stable $\infty$-category of D-modules on stacks as developed for example in \cite{gaitsbook}. We expect such theory with transfer to be useful in constructing categorifications of Hall algebras and quantum groups as explained in \cite{ourGeometricHall2}.

Let $\Spaces$ be an $(\infty,1)$-category closed on limits. The data of a $d$-Segal object in $\Spaces$ is a simplicial object \[
S:\OrdSet^{op}\to\Spaces
\]
satisfying certain combinatorially defined conditions arising from the geometry of $d$-dimensional polytopes.

The main motivations for the definition is that in the $d=2$ case a 2-Segal object gives rise to an associative algebra object in the category of correspondences (or spans) $\Corr(\Spaces)$.

In this article we show that this observation can both be made precise and be extended to the $d>2$ case. We do this by constructing a system of data which we call the "Hall Algebra" $H(S)$ of a simplicial object $S$ and showing that the $d$-Segal conditions on $S$ are equivalent to requiring that $H(S)$ be a $(d-1)$-lax version of a non-unital $\Ainfinity$-algebra object.

Namely, we show the following:

\theoremstyle{plain}
\newtheorem*{MainTheorem}{Theorem}
\begin{MainTheorem}
$S$ is a $d$-Segal object in $\Spaces$ iff $H(S)$ defines a $(d-1)$-lax $\Ainfinity$-algebra object in the category of correspondences in $\Spaces$.
\end{MainTheorem}

In \autoref{sec:Ainfinity} we give a definition an $\Ainfinity$-algebra object in a monoidal $(\infty,1)$-category by providing a datum of compatible associators which are indexed by certain canonical cubes in the category of ordered sets (\autoref{Def:AInfinityDatum}). 
We show in \autoref{Prop:AinfinityDatumLurie} that this datum defines a "usual" $\Ainfinity$-algebra object as defined for example in \cite[Definition 4.1.3.16]{LurieAlgebra}. We generalize the above point of view to define what we call $k$-lax $\Ainfinity$-algebra objects, i.e. we only require invertibility of the associator cubes starting from dimension $k+1$. In this language a usual non-unital $\Ainfinity$-algebra is a $1$-lax $\Ainfinity$-algebra.
The advantage of the packaging  of the data that we propose in \autoref{Def:AInfinityDatum} is that the compatibility conditions on this "cubical" data are straightforward to formulate and are naturally connected to the higher Segal conditions as we show in the proof of \autoref{thm:HigherSegalLax}.

In \autoref{sec:CombinatorialHall}, with the above point of view in mind, we construct the Hall algebra data in two steps: first combinatorially as a system of associator cubes of correspondences in  $\sset^{op}$ and then by applying the Kan extension of $S_{\bullet}$ obtaining a system of cubes of correspondences in $\Spaces$. 

Recall that the category of correspondences in $\Spaces$ has the same space of objects and the space of 1-morphisms between the objects $A,B \in \Spaces$ is given by the space of diagrams (called correspondences or spans)
\[
\stik{1}{
{} \& X_{AB} \ar{dl} \ar{dr} \& {} \\
A \& {} \& B
}
\]
The composition of correspondences is given by taking pullbacks:
\begin{equation*}
\stik{1}{
{} \&{} \& X_{ABC} \ar{dr} \ar{dl} \ar[dd, phantom, "\wedge", very near start]\\
{} \& X_{AB} \ar{dr} \ar{dl} \& {} \& X_{BC} \ar{dr} \ar{dl} \\
A \& {} \& B \& {} \& C
}
\end{equation*}

One can also define higher morphisms in the category of correspondences in a natural way. It was shown in \cite{HaugsengSpans} that this data can be used to construct $(\infty,d)$-categories of correspondences for all $d$. 

The $1$-dimensional associator cube in $\Corr(\sset^{op})$ is the correspondence
\[
\includegraphics[scale=0.5]{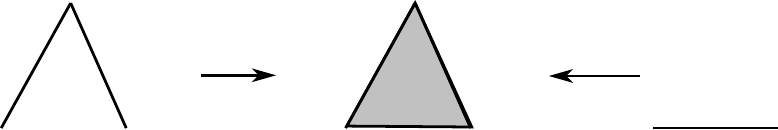}
\]
To connect this to a familiar example note that if we took $S_\bullet$ to be the Waldhausen construction for an abelian category $\CCC$, applying it (or more precisely its canonical extension to $\sset$) to the above would yield the correspondence containing the data for multiplication in the Hall algebra associated to $\CCC$, namely the correspondence
\[
\stik{1}{
\Ob(\CCC)\times \Ob(\CCC) \& \Exact(\CCC) \ar{l}[above]{ends} \ar{r}{mid} \& \Ob(\CCC)
}
\]
\[
ends(0\to U\to V \to W \to 0)=U,W\phantom{MM}
mid(0\to U\to V \to W \to 0)=V
\]
The familiar multiplication in the Hall algebra is obtained from this correspondence via a theory with transfer to vector spaces, see \cite{KapranovDyckerhoff} for details.

The $2$-dimensional associator cube is 
\[
\includegraphics[scale=0.8]{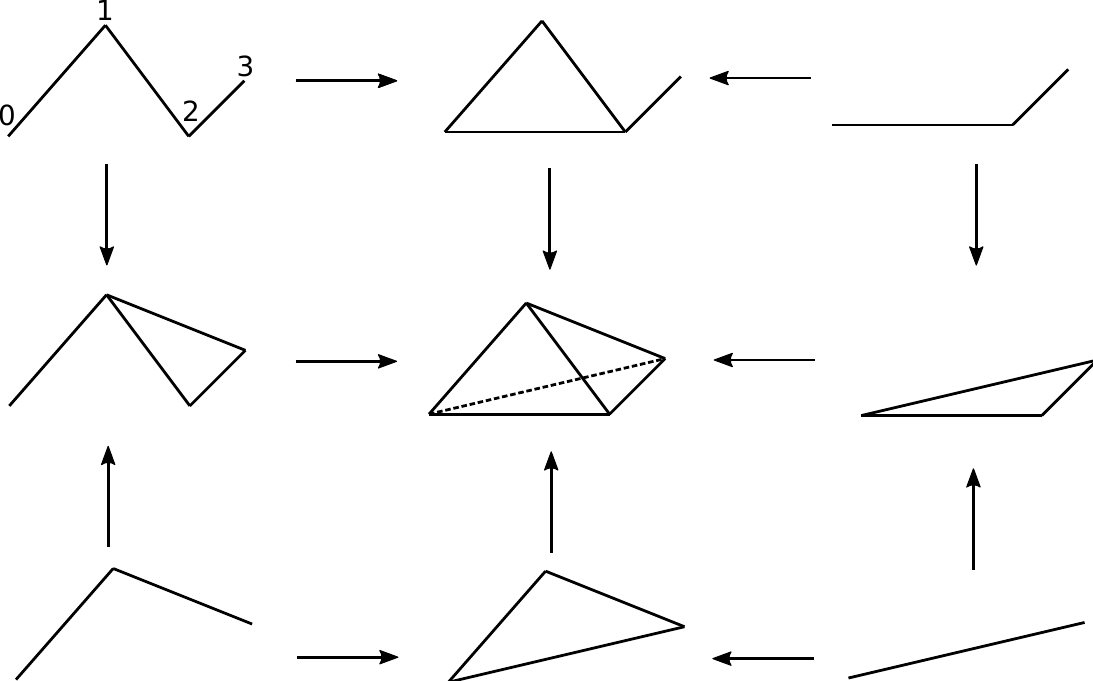}
\]
We call the image of this cube under $S_\bullet$ an \emph{invertible cube of correspondences} if the resulting upper-right and lower-left squares are pullbacks. This definition is natural since the compositions of the sides of the above square are given by pullbacks. This gives conditions on $S$ which are exactly the "upper" and "lower" 2-Segal conditions mentioned in \cite{PoguntkeSegal} and corresponds in the case of the Waldhausen construction to the equivalence between flags and co-flags. 

An observation we make in this article is that the same is true in all dimensions. That is, the higher associators are \emph{invertible $n$-cubes of correspondences} iff the upper and lower $n$-Segal conditions are satisfied. These conditions are defined by certain canonical triangulations of $n$-dimensional polytopes, that can be informally described as prescribing a way for the higher and the lower dimensional polytopes to "fit together". The property of being $n$-Segal can be defined in terms of these two triangulations. Following on the example of the square above, an $n$-cube of correspondences in the category $\Spaces$ is defined to be invertible whenever a certain pair of its subcubes are pullback diagrams in $\Spaces$ (see \autoref{invertiblecubes}). Denote these sub-cubes by $L_n$ and $U_n$. We show that 
\theoremstyle{plain}
\newtheorem*{MainTheorem1}{Theorem}
\begin{MainTheorem1}
$S$ is a lower (resp.upper) $d$-Segal object in $\Spaces$ iff it sends every $L_n$ (resp.$U_n$), $n\geq d$ to a pullback cube in $\Spaces$.
\end{MainTheorem1}
This is \autoref{thm:HigherSegalCubes} and its reformulation \autoref{thm:lowerupperSegal}. This provides a concrete interpretation for the lower and upper Segal conditions in terms of the data of the higher associators constructed in \autoref{sec:CombinatorialHall}. 

\begin{Remark}
The traditional approach to $\Ainfinity$-algebras originates in Stasheff's work \cite{stasheffHomotopy} and calls for encoding higher associativity data using an operadic approach by polytopes called \emph{associahedra}. It was pointed out to us by G.Segal that the idea to use cubes for this  objective was already indicated in the appendix to his work \cite{segalCategories}. The result of our article can be seen as a concrete application of this point of view which connects it to the $d$-Segal spaces approach. The characterization of the $d$-Segal conditions via invertible cubes of correspondences from \autoref{thm:HigherSegalCubes} further suggests 
that a purely cubical approach will be useful in the study of $\infty$-categorical constructions. To our knowledge such a framework does not as of yet exist. For example in \cite{gaitsbook} the authors base some of their fundamental results on the conjectural existence of a cubical model for $(\infty,2)$-categories.
\end{Remark}

The concept of 2-Segal space was independently proposed by Gálvez-Kock-Tonks under the name of \emph{decomposition space} in \cite{decomposition1},\cite{decomposition2},\\ \cite{decomposition3} with an outlook to applications of a more combinatorial nature. The objective of both groups of authors was to define a unified framework for various examples of Hall-algebra like constructions appearing in the literature. Their insight was that the object $S_\bullet$ naturally appears in these examples and that their associativity can be expressed via its properties. As we noted in the beginning of this introduction the main source of examples is the Waldhausen construction. This claim can be made formal as studied in \cite{scheimbauer2Segal}.

The connection between the $2$-Segal conditions and associativity expressed in a various languages was first studied in \cite{KapranovDyckerhoff} and later also in \cite{Penney2Segal}, \cite{WaldeOperads}. The main difference of our approach in that case is that it provides an explicit description of the higher associators via the construction from \autoref{sec:CombinatorialHall}. We note that the origin of this construction is in our article \cite{ourGeometricHall1}.

We would also like to mention the articles \cite{stern2Segal} and \cite{WaldeHigher} that appeared since the publication of the previous version of this article. \cite{WaldeHigher} studies a characterization of higher Segal conditions in terms of certain categorically characterized cubes being pullbacks. These cubes have a significant overlap with the ones that appear in \autoref{thm:lowerupperSegal} and it would be interesting to explore this connection.


\subsection{Notations and technical remarks}

\label{sec:Notations}
In what follows the category $\Spaces$ will denote a complete $(\infty,1)$-category.

In this article we will use two categories of ordered finite sets:

$\AugOrdSet$ - the augmented category of finite ordered sets. The elements of $\AugOrdSet$ will be denoted by \[\ord{0}=\emptyset, \ord{1}=\{0\}, \ord{2}=\{0\rightarrow 1\}, \ord{3}=\{0\rightarrow 1\rightarrow 2\}, \ldots\]

$\OrdSet$ - the category of \emph{nonempty} ordered finite sets. The elements of $\OrdSet$ will be denoted by \[\ordop{0}=\{0\}, \ordop{1}=\{0\rightarrow 1\}, \ordop{2}=\{0\rightarrow 1\rightarrow 2\}, \ldots\]

$\OrdSet^{op}$ appears in the definition of the simplicial object 
\[
S_\bullet: \OrdSet^{op} \rightarrow \Spaces
\]

For our purposes in this note we will always consider the the canonical extension of $S_\bullet$ to $\sset^{op}$, which exists because the category $\Spaces$ is complete. 

By abuse of notation we shall denote the right Kan extension of $S_\bullet$ to $\sset^{op}$ by $S_\bullet$ as well. The above is the reason for the notation we choose for the elements of $\OrdSet^{op}$, i.e. in our constructions in this paper they appear as elements of $\sset^{op}$. 

We will shorten $S_{\ordop{n}}\CCC$ to $S_{n}$.

\subsection{Acknowledgements}
We would like to thank Kobi Kremnitzer for many fruitful and encouraging discussions around this topic. We would also like to thank Andre Henriques, Andy Tonks, Claudia Scheimbauer, Graeme Segal, and Mark Penney for useful comments and discussions related to the subject of this article.

The authors are supported by EPSRC grant [R50311/GA001].

\section{\texorpdfstring{$\Ainfinity$}{A infinity}-algebras and associator cubes}
\label{sec:Ainfinity}
In this section we outline a convenient general point of view on higher associativity data. Namely we introduce a notion of $\AinfinityNU$ object $A_\bullet$ in a monoidal $(\infty,n)$-category. We show that in the case of the $(\infty,1)$-category this notion is equivalent to a usual notion of an $\Ainfinity$-algebra object as defined e.g. in \cite{LurieAlgebra}. 

The $\AinfinityNU$ object is given by a compatible system of $n$-cubes called \emph{associator cubes}.

\subsection{The associator cubes}
\label{sec:AssociativityCubes}

\begin{Definition}
\label{Def:1catncube}
An $n$-cube in an ordinary category $\CCC$ is a diagram indexed by the poset $\{ 0\rightarrow 1\}^n$. 
\end{Definition}
We can easily generalize this to $(\infty,1)$ categories:
\begin{Definition}
\label{Def:AbstractNCube}
Denote by $\square^n$ the simplicial set $\ordop{1}^n$, i.e. the nerve of the poset $\{ 0\rightarrow 1\}^n$ 
\end{Definition}

\begin{Definition}
\label{Def:Infty1ncube}
An $n$-cube in an $(\infty,1)$-category $\CCC$ is a map $\square^n\to\CCC$.
\end{Definition}

\begin{Definition}
\label{Def:degenerateCube}
A cube $C:\square^n\to\CCC$ is called \emph{degenerate} if it factors through some coordinate projection $p:\square^n\to\square^m, m<n$.
\end{Definition}

Let $\CCC$ be a monoidal $(\infty,1)$-category. Conceptually, an associative algebra object $A \in \CCC$ is provided by the datum of multiplication $A \otimes A \rightarrow A$ and a compatible system of higher associators. To provide a precise description of this data let us consider the following family of commutative cubes in the category $\AugOrdSet$:

\begin{Proposition}
\label{associatorCubesProp}
For any $n\geq 2$ there is a unique commutative $n$-cube in $\AugOrdSet$ that contains all of the surjections $\ord{j}\to\ord{j-1}$ for all $1\leq j\leq n+1$.
\end{Proposition}

We call this cube the "$n$-dimensional associator cube".
The paths from $\ord{n+1}$ to $\ord{1}$ on the $n$-dimensional associator cube correspond exactly to all ordered ways of bracketing $n$ letters.

\begin{Example}
The $1$-dimensional associator cube is the arrow
\[
\stik{1}{
\ord{2} \ar[two heads]{r} \& \ord{1}
} \leftrightarrow ((X,Y)\mapsto XY)
\]
\end{Example}

\begin{Example}
The $2$-dimensional associator cube is the square
\[
\stik{1}{
\ord{3} \ar[two heads]{r} \ar[two heads]{d} \& \ord{2} \ar[two heads]{d} \\
\ord{2} \ar[two heads]{r} \& \ord{1}
} \leftrightarrow \left(\alpha: (XY)Z\to X(YZ)\right) 
\]
\end{Example}

\begin{Example}
The $3$-dimensional associator cube is the cube 
\begin{gather*}  
\stik{1}{
 \& \ord{3} \ar[two heads]{rr} \ar[two heads]{dd} \& {} \& \ord{2} \ar[two heads]{dd}\\
\ord{4} \ar[two heads,crossing over]{rr} \ar[two heads]{dd} \ar[two heads]{ur} \& {} \& \ord{3} \ar[two heads]{dd} \ar[two heads]{ur} \& {}\\
{} \& \ord{2} \ar[two heads]{rr} \& {} \& \ord{1}\\
\ord{3} \ar[two heads]{rr} \ar[two heads]{ur} \& {} \& \ord{2} \ar[two heads]{ur}
\latearrow{commutative diagrams/crossing over,commutative diagrams/two heads}{2-3}{4-3}{}
}
\\
\updownarrow\\
\stik{1}{
{} \& (X(YZ))W \ar{r}{\alpha_{{}_{X,Y\cdot Z,W}}}\&  X((YZ)W) \ar{dr}{X\cdot \alpha}\\
((XY)Z)W \ar{ur}{\alpha\cdot W} \ar{dr}{\alpha_{{}_{X\cdot Y,Z,W}}} \& {} \&  {} \& X(Y(ZW))\\
{} \& (XY)(ZW) \phantom{X}\ar[shorten <=-1em,shorten >=-1em]{r}{\Id_{XY}\cdot\Id_{ZW}} \& \phantom{X}(XY)(ZW) \ar{ur}{\alpha_{{}_{X,Y,Z\cdot W}}}
}
\end{gather*}

Note that this diagram is the familiar pentagon identity diagram, with an extra edge (the bottom middle) which is $\Id$.

\end{Example}

\begin{proof}[Proof of \autoref{associatorCubesProp}]
Denote by $p^m_i$ for $i=1,\ldots, m-1$ the surjection $\ord{m}\to\ord{m-1}$ which sends $i,i+1$ to $i$. 

Now fix $n$. Each vertex of an $n$-cube is indexed by a sequence $\epsilon=(\epsilon_1,\ldots,\epsilon_n)$ where $\epsilon_i=0,1$.
Let $depth(\epsilon)=\sum \epsilon_i$. Our cube is given as follows:
\begin{itemize}
    \item At vertex $\epsilon$ we put the object $\ord{n+1-depth(\epsilon)}$.
    \item At each vertex $\epsilon$ let $m=n+1-depth(\epsilon)$ and order the arrows going out of it lexicographicaly. We assign to the arrows the maps $p^{m}_i$ according to this order.
\end{itemize}
It is straightforward to check that this cube commutes. It is unique (up to reordering the indices) by induction on $depth(\epsilon)$:

At depth 0 we must have exactly one vertex assigned $\ord{n+1}$ and $n$ arrows leaving it, so each arrow must be assigned a different $p^{n+1}_i$. This fixes an order on the indices.

Suppose the claim is true through depth $k$. Consider any 2 vertices at depth $k+1$. They are contained in a unique square starting in depth $k$ and ending in depth $k+2$. Suppose the maps from the depth $k$ vertex are $p^{n+1-k}_i,p^{n+1-k}_j$ with $j>i$. If $j>i+1$ then the maps from depth $k+1$ to $k+2$ must be $p^{n+1-(k+1)}_{j-1}$ and $p^{n+1-(k+1)}_{i}$ respectively. If $j=i+1$, then the maps must both be $p^{n+1-(k+1)}_{i}$. In any case they are determined uniquely by the previous depth and we are done.
\end{proof}

Note the following easily verified fact:

\begin{Lemma}
\label{Lem:AssociatorCubeBoundary}
All cubical faces in the boundary of the $n$-dimensional associator cube are ordered disjoint unions of lower dimensional associator cubes and their degeneracies (as defined in \autoref{Def:degenerateCube}).
\end{Lemma}

\begin{Example}
The boundary of the 2-dimensional associator cube decomposes as follows:
\[
\stik{1}{
{} \& {} \& \ord{2}\ar[two heads]{r}[below,yshift=-1.5em]{\bigsqcup}\&\ord{1}\\
{} \& {} \& \ord{1}\ar[equals]{r}\& \ord{1}\\
\ord{1} \ar[equals]{d}[right,xshift=1.7em]{\bigsqcup} \& \ord{2}\ar[two heads]{d} \& \ord{3} \ar[two heads]{r} \ar[two heads]{d} \& \ord{2} \ar[two heads]{d} \\
\ord{1} \& \ord{1} \& \ord{2} \ar[two heads]{r} \& \ord{1}
}
\]
\end{Example}

\begin{Definition}
\label{Def:AInfinityDatum}
Let $\CCC$ be a monoidal $(\infty,1)$-category. An $\AinfinityNU$-object $A_\bullet$ in $\CCC$ is a system of $n$-cubes $A_n$ in $\CCC$ which are compatible with respect to \autoref{Lem:AssociatorCubeBoundary}. That is, the boundary of $A_n$ is composed via the monoidal product in $\CCC$ of lower dimensional $A_m$ and their degeneracies as prescribed by the decomposition of the faces of the $n$-dimensional associator cube.
\end{Definition}

In the rest of the section we prove the following:

\begin{Theorem}
\label{Prop:AinfinityDatumLurie}
An $\AinfinityNU$ object $A_\bullet$ in $\CCC$ gives rise to a non-unital $\Ainfinity$-algebra in $\CCC$ in the sense of \cite[Definition 4.1.3.16]{LurieAlgebra}.
\end{Theorem}
For the proof of \autoref{Prop:AinfinityDatumLurie} the following description of the nerves of  $n$-dimensional associator cubes will be useful

\begin{Remark}
$\square^n$ can also be realized as the nerve of the power set poset of $\{1,2,\ldots,n\}$. 

We have a natural isomorphism $\square^n\to(\square^n)^{op}$ given by $J\mapsto \{1,\ldots,n\}\setminus J$. As a result we may equivalently describe an $n$-cube by giving a map from $(\square^n)^{op}$.
\end{Remark}
\begin{Construction}
\label{Cons:AssociatorCubeNerve}
The nerve of the $n$-dimensional associator cube is given the map $A_n:(\square^n)^{op}\to N(\AugOrdSet)$ defined as follows: 
\begin{itemize}
    \item $A_n$ sends $J\subseteq\{1,\ldots,n\}$ to $\ord{|J|+1}$. 
    \item Given an inclusion $I\subset I\cup \{j\}$, $A_n$ sends it to the map $\ord{|I|+2}\to\ord{|I|+1}$ given by \[
    k \mapsto \left\{\nnstik{1}{
    k \& k<j\\
    k-1 \&  k\geq j
    }\right.
    \]
\end{itemize}
\end{Construction}

\subsubsection{Proof of \autoref{Prop:AinfinityDatumLurie}}
Let us recall some details and terminology from \cite[\S 4.1]{LurieAlgebra}.

A monoidal $(\infty,1)$-category $\CCC$ uniquely defines a \emph{planar $\infty$-operad} $\CCC^{\otimes}$, i.e. a fibration $\CCC^{\otimes}\rightarrow N(\OrdSet^{op})$ with certain properties defined in loc.sit.

\begin{Definition}
Let $\OrdSet_s$ be the subcategory of $\OrdSet$ spanned by all monomorphisms.
\end{Definition}

The imbedding $\OrdSet_s\hookrightarrow\OrdSet$ endows $N(\OrdSet_s^{op})$ with the structure of a planar $\infty$- operad.

\begin{Definition}
A non-unital $\Ainfinity$-algebra in a monoidal category $\CCC$ is a morphism of planar $\infty$-operads $N(\OrdSet_s^{op})\to\CCC^{\otimes}$. 
\end{Definition}
Denote the category of all non-unital $\Ainfinity$ algebras in $\CCC$ by $\AlgOper_{\Ainfinity}(\CCC)$.

The above structure can be constructed as a limit in the following way:

\begin{Definition}
(\cite[\S 4.1.4.1]{LurieAlgebra})
Let $\tau_n$ denote the subcategory of $N(\OrdSet_s)$ spanned by the objects $\ordop{1},\ldots,\ordop{n}$.
\end{Definition}

\begin{Definition}(\cite[\S 4.1.4.2]{LurieAlgebra})
Let $n>0$. A non-unital $\mathbb{A}_n$-algebra object of a monoidal category $\CCC$ is a functor $\tau_n^{op} \rightarrow \CCC^\otimes$ with the following properties:
\begin{itemize}
    \item The diagram 
    \[
    \stik{1}{
     \tau_n^{op}\ar{d}\ar{r} \& \CCC^{\otimes} \ar{dl} \\
    N(\Delta^{op}) 
    }
    \]
    commutes.
    \item The map $\tau^{op}_n\to\CCC^\otimes$ is a map of planar operads. That is, it preserves coCartesian lifts over inert maps in $\OrdSet^{op}$.
\end{itemize}
\end{Definition}
Denote the category of non-unital $\mathbb{A}_n$ algebras in $\CCC$ by $\AlgOper_{\mathbb{A}_n}(\CCC)$.

The inclusion maps $\tau_1^{op}\hookrightarrow \tau^{op}_2\hookrightarrow \tau_n^{op}\hookrightarrow\ldots\hookrightarrow N(\OrdSet_s^{op})$ induce a map\[
\AlgOper_{\Ainfinity}(\CCC)\to\lim_n{ \AlgOper_{\mathbb{A}_n}(\CCC)}
\]
which by \cite[Proposition 4.1.4.9]{LurieAlgebra} is an equivalence. This means that to construct an $\Ainfinity$-algebra we need to construct a compatible system of $\mathbb{A}_n$-algebras. To prove \autoref{Prop:AinfinityDatumLurie} we will explain how the associator cubes allow us to explicitly carry out such a construction.

The connection between $\tau_n$ and cubes is given by the following diagram in $N(\OrdSet)$:

\begin{Definition}
(\cite[Construction 4.1.5.2]{LurieAlgebra}) The \emph{fundamental $n$-cube} $F_n$ is the map $\square^n\to N(\OrdSet)$, defined as follows: 

\begin{itemize}
    \item $F_n$ sends $J\subseteq\{1,\ldots,n\}$ to $\ordop{|J|+1}$.
    \item For an inclusion $I=\{i_1<i_2<\ldots<i_{a-1}\}\subseteq J=\{j_1<j_2<\ldots<j_{b-1}\}$ the map $F_n(I\subseteq J):\ordop{a}\to\ordop{b}$ is defined by
    \[
    m \mapsto \left\{\nnstik{1}{
    0 \& m=0\\
    m' \&  0<m<a, i_m=j_{m'}\\
    b \& m=a
    }\right.
    \]
\end{itemize}

\end{Definition}

From the construction, $F_n$ lands in $\tau_{n+1}$. 
Denote by $\gamma_n$ the induced map $\square^n\to\tau_{n+1}$.

If we denote by $\partial \square^n$ the boundary of $\square^n$ we can see that $\gamma_n$ sends it into the subsimplicial set $\tau_{n+1}^\circ$ which is the maximal subsimplicial set of $\tau_{n+1}$ not containing the edge $(\ordop{1}\to\ordop{n+1}:0\mapsto 0, 1\mapsto n+1)$ .

\begin{Proposition}\emph{(\cite[Proposition 4.1.5.7]{LurieAlgebra}).}
\label{PropLurie:TauPushout}

 The following square is a pushout:
 \[
  \stik{1}{
  \partial \square^n \ar{r} \ar{d}{\gamma_n} \& \square^n \ar{d}{\gamma_n} \\
  \tau_{n+1}^\circ \ar{r} \& \tau_{n+1}
  }
 \]
\end{Proposition}

Since what we are interested in are maps from $N(\OrdSet^{op})$ we note that the square of opposite categories is a pushout as well.
 
\begin{Proposition}\emph{(\cite[Proposition 4.1.5.8]{LurieAlgebra}).}

\label{PropLurie:InessentialExtension}
For any monoidal category $\CCC$ the inclusion $(\tau_n)^{op}\hookrightarrow(\tau_{n+1}^\circ)^{op}$ induces an equivalence of the respective categories of maps of planar $\infty$-operads to $\CCC$.
\end{Proposition}

These two propositions imply that we can build up an non-unital $\Ainfinity$ algebra by giving a compatible system of cubes in $\CCC$. Namely, having defined our functor on $(\tau_n)^{op}$, we can extend it to $(\tau_{n+1}^\circ)^{op}$ via \autoref{PropLurie:InessentialExtension}.

We then have a square
\[
\stik{1}{
  (\partial \square^n)^{op} \ar{r} \ar{d} \& (\square^n)^{op} \ar{d} \\
  (\tau_{n+1}^\circ)^{op} \ar{r} \& \CCC
  }
\]
and this square needs to commute in order to induce a map $(\tau_{n+1})^{op}\to\CCC$ using \autoref{PropLurie:TauPushout}.

We claim that an $\AinfinityNU$ object $A_\bullet$ from \autoref{Def:AInfinityDatum} satisfies this compatibility requirement.
The main observation is that the fundamental cube factors through the nerve of the associator cube from \autoref{Cons:AssociatorCubeNerve} using the following: 

\begin{Definition}
\label{Def:Augmentation}
Let $\aug{-}:\AugOrdSet\to\sset^{op}$
be the contravariant functor defined by the formula
\[
X\mapsto\Hom_{\AugOrdSet}(X,0\rightarrow 1)
\]
i.e. $\aug{X}$ is the simplicial set represented by the ordered set $\Hom_{\AugOrdSet}(X,0\rightarrow 1)$.
We note that on the level of objects it sends  $\ord{n}\in\AugOrdSet$ to the $n$-simplex $\ordop{n}\in\sset^{op}$.
\end{Definition}

We call $\aug{-}$ the \emph{augmentation} functor. It will also play a central role in our construction of the Hall algebra data in \autoref{Hcomb}.

The following is easily checked:

\begin{Proposition}
\label{Prop:SimplerFundamental}
The following diagram commutes:
\[
\stik{1}{
(\square^n)^{op} \ar{r}{(F_n)^{op}} \ar{d}[left]{A_n} \& N(\OrdSet^{op}) \\
N(\AugOrdSet) \ar{ur}[below,xshift=2em]{N(\aug{-})}
}
\]
\end{Proposition}

Recall from \autoref{Lem:AssociatorCubeBoundary} that the boundary of the associator cube is built out of ordered disjoint unions of simpler pieces, where by "simpler pieces" we mean degeneracies of lower dimensional associator cubes.

It is straightforward to verify that the image of an ordered disjoint union of cubes under $\aug{-}$ is made up of what's called \emph{decomposable simplices} in $\tau^\circ_n$ in the proof of \cite[Proposition 4.1.5.8]{LurieAlgebra}. From this same proof we see that the image of a decomposable simplex is determined in an essentially unique way by the pieces it is composed of. Therefore the condition of \autoref{Def:AInfinityDatum} implies the compatibility required to extend from $(\tau_n)^{op}$ to $\tau_{n+1}^{op}$. This finishes the proof.

\begin{Remark}
The advantage of packaging of the associativity data via the associator cubes is that it gives a concrete system of compatibilities within the data, which is useful for comparing them with the higher Segal conditions below.
\end{Remark}

\subsection{Lax non-unital \texorpdfstring{$\Ainfinity$}{A infinity}-algebras}
\label{sec:LaxAInfinity}
Here we propose a generalization of non-unital $\Ainfinity$-algebras to the situation where not all of the data provided has invertible higher morphisms. The idea is to generalize \autoref{Def:AInfinityDatum}.  \autoref{Def:LaxAinfinityAlgebra} requires some technical background.
The main thing we need is a weaker notion of $n$-cube than we used in \autoref{sec:Ainfinity} where not all higher morphisms are invertible.

In \cite{streetParity}, and in more detail in \cite{aitchisonCubes}, there are constructed $n$-cubes as \emph{parity complexes}. Parity complexes are a certain precise way to give generating data for a higher category introduced in \cite{streetParity}. Let us denote these cubes by $\cubeInfinity{n}$.
The main difference from \autoref{Def:AbstractNCube} is that they specify a system of \emph{directed} higher morphisms where the $k$-morphisms appear either as sources or targets for $k+1$ morphisms, but not both as in the $(\infty,1)$-case.

\subsubsection{Lax cubes}
\label{sec:LaxCubes}
The cells of $\cubeInfinity{n}$ correspond to sequences of length $n$ in the symbols $\{-,0,+\}$. The number of $0$'s determines the dimension of the cell, and the $+$'s and $-$'s determine which cells are in its source and which cells in its target and which cells it can compose with.

The basic examples are the interval and the square
\[
\stik{1}{
-\ar{rr}{0}\& {} \& +
}
\]
\[
\stik{1}{
{} \& -- \ar{dr}[above]{0-} \ar{dl}[above]{-0} \& {}\\
-+ \ar{dr}[below]{0+} \ar[Rightarrow,shorten <=2em,shorten >=2em]{rr}{00} \& {} \& +- \ar{dl}[below]{+0}\\
{} \& ++ \& {}
}
\]

Continuing inductively, an $n$-cube has two collections of $n-1$ faces which compose into the source and target of the $n$-morphism which is the $0\ldots 0$ face. The source collection is comprised of the faces $(-0\ldots 0),(0+0\ldots 0),(00-0\ldots 0),\ldots$ and the target is comprised of the faces $(+0\ldots 0),(0-0\ldots 0),\ldots$ (composed in reverse order).

To actually compose these collections together it is necessary for the successive faces in the list to have the same source and target, which requires "whiskering" with lower dimensional faces. For the complete details of this procedure we refer to \cite[\S 4]{aitchisonCubes}.

\begin{Proposition}
An $n$-cube as describes above prescribes a strict $n$-category.
\end{Proposition}
\begin{proof}
As is proven in \cite{streetParity}, any parity complex (such as the $n$-cube in the way described above) gives rise to a free $\omega$-category. Since the $n$-cube has no non-trivial faces of dimension higher than $n$, the $\omega$-category corresponding to it is in fact a strict $n$-category. 
\end{proof}

Let $\CCC$ be an $(\infty,d)$-category.
By a lax $n$-cube in $\CCC$ we mean a morphism $N(\cubeInfinity{n})\to\CCC$ (where the nerve is taken using the above proposition).

\begin{Definition}
We say that a $k$-cube in $\CCC$ is invertible if the morphism from the source $k-1$ morphism to the target $k-1$ morphism is invertible.
\end{Definition}

The main example in this paper is when $\CCC$ is the $(\infty,n)$-category of correspondences and then an $k$-cube is a $k$-cube of correspondences as described in \autoref{sec:Corr}.
A non-lax (i.e. invertible) $k$-cube in this case would have additional restrictions in the form of certain subcubes in the diagram being pullback cubes.

\subsubsection{Definition of a Lax \texorpdfstring{$\Ainfinity$}{A infinity}-algebra}
\begin{Definition}
\label{Def:LaxAinfinityAlgebra}
Let $\CCC$ be a monoidal $(\infty,n)$ category and fix $d\leq n$.
A $d$-lax $\Ainfinity$ object in $\CCC$ is a system of lax cubes $\mathcal{A}_k$ in $\CCC$ which are compatible as in \autoref{Def:AInfinityDatum} and such that $\mathcal{A}_k$ is invertible for $k>d$.
\end{Definition}

 We will see in \autoref{sec:HigherSegal} that $d$-Segal objects in a category $\CCC$ give rise to examples of $d$-lax $\Ainfinity$ objects in the $(\infty,d)$ category of correspondences in $\CCC$.
 
 \begin{Remark}
 This notion is related to the notion of "skew monoidal category" studied in \cite{slaSkewMonoidal} and \cite{StreetSkewMonoidal}. The main difference is that skew monoidal categories have the additional data and requirements relating to a lax unital structure. One could attempt to add it by adding to the associator cubes also cubes which contain general (non-surjective) maps.
 Two issues then arise: First, it is not immediately clear whether this is equivalent to some finite amount of data in each dimension. Second, in the definition of skew monoidal category the left unit is in the opposite direction to the right unit. This seems to be unnatural in our setting where the units would correspond to the two maps $\ord{1}\hookrightarrow\ord{2}$.
 \end{Remark}

\section{Correspondences}
\label{sec:Corr}
Let $\Spaces$ be an $(\infty,1)$-category with finite limits.
For the constructions in this article we will describe the $(\infty,1)$-category of \emph{correspondences} or \emph{spans} in  $\Spaces$,  $\Corr(\Spaces)$ and discuss its generalizations which allow non-invertible higher morphisms. 

\subsection{Cubes of correspondences}
\label{sec:CorrCubes}
\begin{Definition}
The abstract $n$-cube of corrsepondences is the poset $\CorrCube{n}$ of faces of an $n$-cube. i.e. there is a map from a face $K$ to a face $L$ in $\CorrCube{n}$ when $L$ is a subface of $K$.
\end{Definition}

The notation of \autoref{sec:LaxCubes} can be thought of as indexing the faces of an $n$-cube, so we can use this to denote the objects of this poset by sequences $\epsilon=(\epsilon_1,\ldots,\epsilon_n)$ of $-,0,+$'s. 
There is a map $\epsilon\to\delta$ in $\CorrCube{n}$ exactly when $\epsilon_i=+$ (resp $-$) implies $\delta_i=+$ (resp $-$) for all $i$.

\begin{Example}
The abstract $1$-cube of correspondences is the poset\[
-\leftarrow 0 \rightarrow +
\]
\end{Example}

\begin{Example}
The abstract $2$-cube correspondences is the poset
\begin{equation}
\label{CorrSquare}
\stik{1}{
-- \& 0- \ar{l} \ar{r} \& +- \\
-0 \ar{u} \ar{d}\& 00 \ar{u} \ar{d}\ar{l} \ar{r} \& +0 \ar{u} \ar{d} \\
-+ \& 0+ \ar{l} \ar{r} \& ++ \\
}
\end{equation}
\end{Example}

\begin{Definition}
\label{def:CorrCube}
An $n$-cube of correspondences in an $(\infty,1)$–category $\Spaces$ is a functor $N(\CorrCube{n})\to\Spaces$
\end{Definition}

\begin{Notation}
\label{not:corrsubcube}
Let $C:N(\CorrCube{n})\to\Spaces$. Let $\epsilon\to\delta$ be an arrow in $\CorrCube{n}$, then we denote by $C^\epsilon_\delta$ the $k$-cube in $\Spaces$ which is the image under $C$ of the cube of paths from $\epsilon$ to $\delta$. The dimension $k$ is exactly the difference in dimension between $\epsilon$ and $\delta$ when considered as faces of the abstract $n$-cube.
\end{Notation}

\begin{Example}
Let $C$ be a square of correspondences in $\Spaces$, i.e. a diagram
\[
\stik{1}{
C(--) \& C(0-) \ar{l} \ar{r} \& C(+-) \\
C(-0) \ar{u} \ar{d}\& C(00) \ar{u} \ar{d}\ar{l} \ar{r} \& C(+0) \ar{u} \ar{d} \\
C(-+) \& C(0+) \ar{l} \ar{r} \& C(++) \\
}
\]
then e.g. $C^{00}_{+-}$ is the square
\[
\stik{1}{
C(0-)\ar{r} \& C(+-) \\
C(00) \ar{u} \ar{r} \& C(+0) \ar{u} 
}
\]
\end{Example}

\subsection{Grids of cubes of correspondences}
\label{sec:corrCubeGrids}
Extending the notation of \autoref{sec:LaxCubes}, we can define 
\begin{Definition}
The abstract $n_1\times\cdots\times n_k$ grid of $k$–cubes $\cubeInfinity{[n_1,\ldots,n_k]}$ is the srtict infinity category (in the sense of \cite{streetParity}) that has as faces sequences of length $k$ in the symbols $\cInt{m}{l},l\geq 0,m\geq 0$ such that in place $i$ there is a $\cInt{m}{l}$ symbol with $0\leq l+m\leq n_i$.
\end{Definition}
Intuitively, the symbol $\cInt{m}{l}$ means "an interval of length $l$ starting at $m$".

Note that $\cubeInfinity{[n_1=1,\ldots,n_k=1]}\cong\cubeInfinity{k}$ with the replacements $$"\cInt{0}{0}"\mapsto "-","\cInt{1}{0}"\mapsto "+","\cInt{0}{1}"\mapsto "0"$$.

\begin{Example}
$\cubeInfinity{[2]}$ is \[
\stik{1}{
\cInt{0}{0} \ar{r}[above]{\cInt{0}{1}} \ar[bend left]{rr}[above]{\cInt{0}{2}}\& \cInt{1}{0} \ar{r}[above]{\cInt{1}{1}}\& \cInt{2}{0}
}
\]
\end{Example}
\begin{Example}
The 0 and 1 faces of $\cubeInfinity{[2,1]}$ assemble into the diagram
\[
\stik{1}{
\cInt{0}{0}\cInt{0}{0}\ar{d}[left]{\cInt{0}{0}\cInt{0}{1}} \ar{r}[above]{\cInt{0}{1}\cInt{0}{0}} \ar[bend left]{rr}[above]{\cInt{0}{2}\cInt{0}{0}}\& \cInt{1}{0}\cInt{0}{0} \ar{d}[left]{\cInt{1}{0}\cInt{0}{1}} \ar{r}[above]{\cInt{1}{1}\cInt{0}{0}}\& \cInt{2}{0}\cInt{0}{0} \ar{d}[right]{\cInt{2}{0}\cInt{0}{1}} \\
\cInt{0}{0}\cInt{1}{0} \ar{r}[below]{\cInt{0}{1}\cInt{1}{0}} \ar[bend right]{rr}[below]{\cInt{0}{2}\cInt{1}{0}}\& \cInt{1}{0}\cInt{1}{0} \ar{r}[below]{\cInt{1}{1}\cInt{1}{0}} \& \cInt{2}{0}\cInt{1}{0} 
}
\]
\end{Example}

Proceeding in the same way as \autoref{sec:CorrCubes} we define:
\begin{Definition}
The abstract $n_1\times\cdots\times n_k$ grid of cubes of correspondences $\CorrCube{[n_1,\ldots,n_k]}$ is the poset of faces of $\cubeInfinity{[n_1,\ldots,n_k]}$.
\end{Definition}

\begin{Example}
\label{Ex:CorrCube2}
$\CorrCube{[2]}$ is the poset 
\[
\stik{1}{
{} \&{} \& \cInt{0}{2} \ar{dr} \ar{dl} \ar[dd, phantom, "\lrcorner"{anchor=center, rotate=135,scale=2}, very near start]\\
{} \& \cInt{0}{1} \ar{dr} \ar{dl} \& {} \& \cInt{1}{1} \ar{dr} \ar{dl} \\
\cInt{0}{0} \& {} \& \cInt{1}{0} \& {} \& \cInt{2}{0}
}
\]
\end{Example}

\subsection{The \texorpdfstring{$(\infty,N)$}{(infinity,N)}-category of correspondences}
\label{secNUppleCorr}
Cubes of correspondences constructed above can be used to define $(\infty,N)$-categories of correspondences. For $d=1$ the definition appears in several places in the literature. For $d>1$ such categories were defined by Haugseng in \cite{HaugsengSpans}. These definitions use the realization of $(\infty,1)$-categories as complete Segal spaces and a generalization of these for $N>1$ due to Barwick \cite{BarwickSegal}. We recall these concepts in \autoref{sec:uppleSegal}. In the present section we will recall the construction of the $(\infty,N)$-category of correspondences $\Corr_N$ using the formalism of cubes of correspondences defined above, i.e. in a slightly different language from \cite{HaugsengSpans}. 

An important feature of this construction is that if we want to consider the data of non-invertible higher morphisms in the category of correspondences, it naturally organizes into a \emph{$N$-upple Segal space} (i.e. a kind of a higher dimensional analog of a double category - see \autoref{sec:uppleSegal}). The advantage of the $N$-upple language is that it allows for more general cells to appear, i.e. general grids of cubes vs only those where in all but one direction we have equivalences. However the theory of these objects is not yet well developed in the literature. Therefore we state \autoref{thm:HigherSegalLax} in the framework of $n$-fold Segal Spaces. To do this we need to pass from $N$-upple to $(\infty,N)$-categories. The main fact we need is that there is a canonical functor $\HaugsengU$ defined in \cite{HaugsengSpans} which assigns an $(\infty,N)$-category to an $N$-upple Segal space.

Let $\Spaces$ be a category. In \cite{HaugsengSpans} there is constructed an $N$-uple Segal space $\CORR_N(\Spaces)$ (called $\SPAN_N$ in loc. cit.). For the sake of consistency we reproduce his definition below in the language of grids of cubes:
\begin{Definition}
Let $\CorrCubeSegal{[n_1,\ldots,n_N]}$ be the family of subposets of $\CorrCube{[n_1,\ldots,n_N]}$ the elements of which are identical except at one place in the sequence. i.e. those of the form in \autoref{Ex:CorrCube2}.
\end{Definition}

\begin{Definition}
\label{def:NuppleCorr}
$\CORR_N(\Spaces)$ is the $N$-upple Segal space with $(n_1,\ldots,n_N)$ space equal to $\MapsSegal(N(\CorrCube{[n_1,\ldots,n_N]}),\Spaces)$, where $\MapsSegal$ is the subspace of $\Maps$ consisting of maps that preserve pullbacks when restricted to any member of $\CorrCubeSegal{[n_1,\ldots,n_N]}$.
\end{Definition}

\begin{Example} The \texorpdfstring{$N=1$}{N=1} case:
\label{sec:Corr1}
The (0) space is given by mapping from $\CorrCube{0}$ which is trivial, so this is just the space of objects of $\Spaces$.

The (1) space is given by mapping from $\CorrCube{1}$ and hence is the space of correspondences
\[
\stik{1}{
{} \& X_{AB} \ar{dl} \ar{dr} \& {} \\
A \& {} \& B
}
\]

The (2) space, which gives composition, is given by mapping from $\CorrCube{2}$ while respecting limits, and hence is the space of diagrams

\[
\stik{1}{
{} \&{} \& X_{ABC} \ar{dr} \ar{dl} \ar[dd, phantom, "\lrcorner"{anchor=center, rotate=135,scale=2}, very near start]\\
{} \& X_{AB} \ar{dr} \ar{dl} \& {} \& X_{BC} \ar{dr} \ar{dl} \\
A \& {} \& B \& {} \& C
}
\]
and so on. Note that the universality of pullbacks is what implies the Segal conditions here.

\end{Example}

\begin{Definition}
Denote $\Corr_N(\Spaces):=\HaugsengU(\CORR_N(\Spaces))$ to be the corresponding $N$-fold Segal space.
\end{Definition}

\begin{Remark}
It is shown in \cite{HaugsengSpans} that $\Corr_N(\Spaces)$ is complete, i.e. it is in $\CAT{n}$.
\end{Remark}

\subsection{Invertible cubes of correspondences}
Recall from \autoref{Seq:UnderlyingCategoryFunctor} that for an $(\infty,N)$ category $\CCC$ presented as an $N$-fold Segal space we have its underlying $(\infty,N-1)$ category $\CCC^{\CAT{(N-1)}}\hookrightarrow \CCC$ obtained by discarding non-invertible $N$-cells. Let us analyze this in the case of $\CCC=\Corr_N(\Spaces)$.

\subsubsection{The \texorpdfstring{$N=2$}{N=2} case}
For simplicity let us start from considering the case $N=2$. 

The $1,1$ cells of $\CORR_2(\Spaces)$ are diagrams of the form
\[
\stik{1}{
A \& E \ar{l} \ar{r} \& B \\
F \ar{u} \ar{d}\& Z \ar{u} \ar{d}\ar{l} \ar{r} \& G\ar{u} \ar{d} \\
C \& H \ar{l} \ar{r} \& D \\
}
\]
and so the $1,1$ cells of $\Corr_2(\Spaces)$ are diagrams of the form
\[
\stik{1}{
A \& E \ar{l} \ar{r} \& B \\
F \ar{u}[above,sloped]{\sim} \ar{d}[below,sloped]{\sim}\& Z \ar{u} \ar{d}\ar{l} \ar{r} \& G\ar{u}[below,sloped]{\sim} \ar{d}[above,sloped]{\sim} \\
C \& H \ar{l} \ar{r} \& D \\
}
\]
Such a diagram is invertible (as a map between the top and bottom rows) iff the maps $Z\to E,Z\to H$ are isomorphisms \emph{iff} the upper-right and lower-left squares are pullback squares. This leads to

\begin{Definition}
Say that a square of correspondences is \emph{invertible} if the upper-right and lower-left squares are pullback squares.
\end{Definition}

This leads us to the definition of the following subobject of $\CORR_2(\Spaces)$:
\begin{Definition}
Let $\CORR_{2,1}(\Spaces)$ be the subfunctor - from $\OrdSet^{\times 2}$ to spaces - of $\CORR_2(\Spaces)$ which has $2$-dimensional cells only those where all squares of correspondences involved are invertible.
\end{Definition}

\begin{Proposition}
\label{Prop:Corr21Segal}
$\CORR_{2,1}(\Spaces)$ is a $2$-uple Segal space.
\end{Proposition}

For the proof we need the following definition and lemma from \autoref{sec:pullbacklemma}

\begin{Definition}[\autoref{Def:PBCube}]
A commutative cube is said to be a \emph{pullback} cube if it presents the source vertex as the limit of the rest of the diagram. 
\end{Definition}

\begin{Lemma}[\autoref{Lem:pullbackcube}]
Consider a cube in an $\infty$-category
\[
\stik{1}{
 \& A \ar{rr} \ar{dd} \& {} \& B \ar{dd}\\
X \ar[crossing over]{rr} \ar{dd} \ar{ur} \& {} \& Y \ar{dd} \ar{ur} \& {}\\
{} \& C \ar{rr} \& {} \& D\\
Z \ar{rr} \ar{ur} \& {} \& W \ar{ur}
\latearrow{commutative diagrams/crossing over}{2-3}{4-3}{}
}
\]
And suppose that $\squCorns{A}{B}{C}{D}$ is a pullback square, then $\squCorns{X}{Y}{Z}{W}$ is a pullback square if and only if the whole cube is a pullback cube.
\end{Lemma}

\begin{proof}[Proof of \autoref{Prop:Corr21Segal}]
Checking the Segal conditions comes down to checking that the composition of two invertible squares is invertible, so we need to consider a diagram of the form
\[
\stik{1}{
{} \& {} \& [-2em]F_1\& {} \& {}\\
A_1 \& B_1 \ar{l} \& F_2\& D_1 \ar{r}\& E_1\\
A_2  \ar{u} \ar{d}\& B_2\ar{l} \ar{u} \ar{d}\& F_3\ar[phantom, "C_1"{name=C1,at start,xshift=1em}] \& D_2 \ar{r} \ar{u} \ar{d}\& E_2  \ar{u} \ar{d}\\
A_3 \& B_3 \ar{l} \& \phantom{F_4}\ar[phantom, "C_2"{name=C2,at start,xshift=1em}] \& D_3 \ar{r}\& E_3\\
{} \& {} \&  \phantom{F_5}\ar[phantom, "C_3"{name=C3,at start,xshift=1em}] \& {} \& {}
\ar[from=2-2,to=C1]
\ar[from=3-2,to=C2]
\ar[from=4-2,to=C3]
\ar[from=2-4,to=C1]
\ar[from=3-4,to=C2]
\ar[from=4-4,to=C3]
\ar[from=C2,to=C1]
\ar[from=C2,to=C3]
\ar[from=1-3,to=2-2,crossing over]
\ar[from=2-3,to=3-2,crossing over]
\ar[from=3-3,to=4-2,crossing over]
\ar[from=1-3,to=2-4,crossing over]
\ar[from=2-3,to=3-4,crossing over]
\ar[from=3-3,to=4-4,crossing over]
\ar[from=2-3,to=1-3,crossing over]
\ar[from=2-3,to=3-3,crossing over]
}

\]
and for instance we need to check that the composition
\[
\stik{1}{
F_2\ar{r}\ar{d} \& E_2\ar{d}\\
F_1 \ar{r} \& E_1
}
\]
is a pullback square. by assumption the square $\squCorns{D_2}{E_2}{D_1}{E_1}$ is a pullback, so by the pasting lemma for pullbacks what we want to show is equivalent to
\[
\stik{1}{
F_2\ar{r}\ar{d} \& D_2\ar{d}\\
F_1 \ar{r} \& D_1
}
\]
being a pullback square. Using \autoref{Lem:pullbackcube} and noting that the opposite cube faces $\squCorns{F_2}{B_2}{D_2}{C_2},\squCorns{F_1}{B_1}{D_1}{C_1}$ are pullbacks (by assumption in \autoref{def:NuppleCorr}), this is equivalent to $\squCorns{B_2}{C_2}{B_1}{C_1}$ being a pullback square, which is also true by assumption.
The other checks are identical.
\end{proof}

The above discussion can be summarized by:
\begin{Proposition}
The imbedding $\Corr_2(\Spaces)^{\CAT{1}}\hookrightarrow \Corr_2(\Spaces)$ factors through $\HaugsengU(\CORR_{2,1}(\Spaces))\hookrightarrow\Corr_2(\Spaces)$ and the map $$\Corr_2(\Spaces)^{\CAT{1}}\to\HaugsengU(\CORR_{2,1}(\Spaces))$$ is an equivalence.
\end{Proposition}

Note also that using the imbedding $\CAT{1}\hookrightarrow\CAT{2}$ we have a map $$\Corr_1(\Spaces)\to\Corr_2(\Spaces)$$ which factors through $\Corr_2(\Spaces)^{\CAT{1}}$ by adjointness and we have:
\begin{Proposition}
The map $\Corr_1(\Spaces)\to\Corr_2(\Spaces)^{\CAT{1}}$ is an equivalence.
\end{Proposition}
\begin{proof}
The squares of the imbedding $\Corr_1(\Spaces)\to\Corr_2(\Spaces)$ are those of the form
\[
\stik{1}{
A \& E \ar{l} \ar{r} \& B \\
A \ar[equals]{u} \ar{d}[below,sloped]{\sim}\& E \ar[equals]{u} \ar{d}[below,sloped]{\sim} \ar{l} \ar{r} \& B\ar[equals]{u} \ar{d}[below,sloped]{\sim} \\
C \& H \ar{l} \ar{r} \& D \\
}
\]
and it is clear that this space is canonically a retract of the corresponding space for $\Corr_2(\Spaces)^{\CAT{1}}$
\end{proof}

\begin{Corollary}
$\Corr_1(\Spaces)$ is equivalent to $\HaugsengU(\CORR_{2,1}(\Spaces))$
\end{Corollary}
This means that if we have a map of $2$-uple Segal spaces $\DDD\to\CORR_2(\Spaces)$ which factors through $\CORR_{2,1}(\Spaces)$ then the corresponding map $\HaugsengU(\DDD)\to\Corr_2(\Spaces)$ factors through $\Corr_1(\Spaces)$.

\subsubsection{The \texorpdfstring{$N>2$}{N>2} case}
\begin{Definition}
\label{invertiblecubes}
Using \autoref{not:corrsubcube}, we say that an $N$-cube $C$ of correspondences is \emph{invertible} if the $N$-cubes $C^{00\ldots0}_{+-+-\ldots}$ and $C^{00\ldots0}_{-+-+\ldots}$ are pullback cubes.
\end{Definition}

\begin{Definition}
Let $\CORR_{N,(N-1)}(\Spaces)$ be the sub-$\OrdSet^{\times N}$ space of $\CORR_N(\Spaces)$ which has $N$-dimensional cells only those where all cubes of correspondences involved are invertible.
\end{Definition}

An identical line of reasoning to the $N=2$ case, using \autoref{Cor:pullbackNcube} yields:

\begin{Theorem}
\label{Prop:CORRInvertibleCorr}
$\Corr_{(N-1)}(\Spaces)$ is equivalent to $\HaugsengU(\CORR_{N,(N-1)}(\Spaces))$
\end{Theorem}
As before, this means that if we have a map of $N$-uple Segal spaces $\DDD\to\CORR_N(\Spaces)$ which factors through $\CORR_{N,(N-1)}(\Spaces)$ then the corresponding map $\HaugsengU(\DDD)\to\Corr_N(\Spaces)$ factors through $\Corr_{(N-1)}(\Spaces)$.

Extending this inductively we can define $\CORR_{N,k}(\Spaces)$ which is equivalent to $\CORR_k(\Spaces)$ for $k\leq N$, and get

\begin{Corollary}
\label{thm:CorrFactors}
Suppose we have a map of $N$-uple Segal spaces $\DDD\to\CORR_N(\Spaces)$ which factors through $\CORR_{N,k}(\Spaces)$, then $\HaugsengU(\DDD)\to\Corr_N(   \Spaces)$ factors through $\Corr_k(\Spaces)$.
\end{Corollary}

\section{Construction of Hall algebra data}
\label{sec:CombinatorialHall}

Our goal in this section is to construct, starting from a simplicial object $S_\bullet$, the data of an $\Ainfinity$–algebra (as in \autoref{sec:Ainfinity}) which we call the \emph{Hall algebra} of $S_\bullet$. In \autoref{sec:HigherSegal} we then give a precise criterion for when this data is associative to various degrees.

\begin{Notation}
\label{def:SimplicialObjectInSpaces}
Let $\Spaces$ be an $(\infty,1)$-category which admits small limits and consider a simplicial object $S_\bullet \in \Spaces^{\OrdSet^{op}}$ which sends $\ordop{0}$ to the final object of $\Spaces$.
\end{Notation}
\begin{Remark}
Note that the simplicial objects given by the Waldhausen construction and it's higher dimensional generalizations as described in \cite{KapranovDyckerhoff}, \cite{PoguntkeSegal} and \cite{DyckerhoffJasso} satisfy the above condition.
\end{Remark}
The product of the Hall algebra (i.e. the image of the 1-dimensional associator cube) is given by the correspondence
\begin{equation}
\label{eq:MultCorr}
\stik{1}{
{} \& S_2 \ar{dl}[above, xshift=-0.7em]{\delta_0\times\delta_2} \ar{dr}{\delta_1} \& {} \\
S_1\times_{S_0} S_1 \& {} \& S_1
}
\end{equation}
where the maps $\delta_i$ are the face maps.  

We next need to construct a square with boundary

\begin{equation}
\label{eq:assocsquareBoundary}
\stik{1}{
S_1^3 \& S_2\times_{S_0} S_1 \ar{l} \ar{r} \& S_1\times_{S_0} S_1 \\
S_1\times_{S_0} S_2 \ar{u} \ar{d} \& {} \&  S_2 \ar{u} \ar{d} \\
S_1\times_{S_0} S_1 \& S_2 \ar{l} \ar{r} \& S_1
}
\end{equation}

i.e. we need to construct a square
\[
\stik{1}{
S_1^3 \& S_2\times S_1 \ar{l} \ar{r} \& S_1\times S_1 \\
S_1\times S_2 \ar{u} \ar{d} \& X\ar{u} \ar{d} \ar{l} \ar{r} \&  S_2 \ar{u} \ar{d} \\
S_1\times S_1 \& S_2 \ar{l} \ar{r} \& S_1
}
\]

The natural object to put in the middle is $S_3$, and we will see that the requirement that this square is invertible is equivalent to a subset of the 2-Segal conditions arising from the triangulations of a square.

In order to give an explicit construction for all dimensions, the general construction will be work as follows: We will describe a combinatorial construction of cubes of correspondences in $\sset^{op}$, and then apply $S_\bullet$ to get a cube of correspondences in $\Spaces$.
Since it is no extra work, as well as describing the cubes corresponding to the associator cubes, we will construct cubes corresponding to any $n$-cube in $\AugOrdSet$.

\subsection{Construction in \texorpdfstring{$\sset^{op}$}{sset op}}
\label{Hcomb}
Here we describe a construction of a system of $n$-cubes of correspondences in $\sset^{op}$, indexed by $n$-cubes in $\AugOrdSet$.

Recall the augmentation map $\aug{-}:\AugOrdSet\to\sset^{op}$ from \autoref{Def:Augmentation}. The following is central to our definition of $H_{comb}$:
\begin{Lemma}
\label{Cons:AugmentationExtension}
Suppose $X$ is an interval of a linearly ordered set $Y$. Let $i:X\to Y$ be the inclusion and $j:Y\setminus X\to Y$ be the inclusion of the complement. Then there is a unique map $i_!:\aug{X}\to\aug{Y}$ such that $j^*i_!$ is constant (as a map from $\aug{X}$ to $\aug{Y\setminus X}$) and $i^*i_!=\Id$.
\end{Lemma}

\begin{proof}
Considering the constant maps $0,1:X\to (0\to 1)$ we see that for any $\varphi\in\aug{X}$, $i_!$ must be $0$ on all elements of $Y$ below $X$ and $1$ on all elements of $Y$ above $X$. It is immediate that this defines a map of ordered sets satisfying the requirements of the lemma.
\end{proof}

\begin{Definition}
\label{Def:HCombOfGenerealMaps}
Let $\alpha:X\to Y$ be a map of ordered sets. Define $H_{comb}(\alpha)$ to be the sub simplicial set in $\aug{X}$ generated by the imbeddings of $\aug{(\alpha)^{-1}(y)}$ (using \autoref{Cons:AugmentationExtension}) for all $y\in Y$.
\end{Definition}

\begin{Construction}
\label{Lem:GeneralCombConstruction}
Consider a diagram of maps in $\AugOrdSet$
\[
X\xrightarrow{f}Y\xrightarrow{g}Z
\]
and let $\alpha=g\circ f$. Then we have maps $H_{comb}(f)\to H_{comb}(\alpha)$, $H_{comb}(g)\to H_{comb}(\alpha)$ given as follows:

\flushleft{\textbf{Case 1}: $H_{comb}(f)\to H_{comb}(g\circ f)$}

$H_{comb}(f)$ is generated by $\aug{f^{-1}(y)},y\in Y$. Let $z\in Z$. For any $y\in g^{-1}(z)$ the ordered set $f^{-1}(y)$ is a sub-interval in the interval $\alpha^{-1}(z)$.
Therefore $\aug{f^{-1}}(y)\subset\aug{\alpha^{-1}(z)}$ where both are considered inside $\aug{X}$ via \autoref{Cons:AugmentationExtension}. More precisely, the following diagram commutes
\[
\stik{1}{
\aug{f^{-1}}(y) \ar{d} \ar[hookrightarrow]{dr} \& {}\\
\aug{\alpha^{-1}(z)} \ar[hookrightarrow]{r} \& \aug{X}
}
\]
where the maps are the ones from \autoref{Cons:AugmentationExtension}.

\flushleft{\textbf{Case 2}: $H_{comb}(g)\to H_{comb}(g\circ f)$}

Recall that $H_{comb}(g)$ is generated inside $\aug{Y}$ by $\aug{g^{-1}(z)},z\in Z$ and $H_{comb}(\alpha)$ is generated inside $\aug{X}$ by $\aug{\alpha^{-1}(z)}=\aug{f^{-1}(g^{-1}(z))}$.
Therefore $f$ gives a map $\alpha^{-1}(z)\to g^{-1}(z)$ which by contravariance of $\aug{-}$ gives our desired map.
\end{Construction}

We can now construct the image of $H_{comb}$ on $n$-cubes.
\begin{Definition}
\label{def:Hcomb}
To an $n$-cube in $\AugOrdSet$ we perscribe a $n$-cube of correspondences in $\sset^{op}$ as follows:
The centers of the $k$-faces are given by sequences of $k$ composable maps on the given cube using \autoref{Def:HCombOfGenerealMaps}, and the maps to the centers of the corresponding $k+1$ faces are given by \autoref{Lem:GeneralCombConstruction}.
\end{Definition} 
\begin{Proposition}
\label{Prop:HCombCubesConstructionCommutes}
The cubes of correspondences so constructed commute.
\end{Proposition}
Proof in \autoref{ProofOf:Prop:HCombCubesConstructionCommutes}

\subsubsection{Examples for small \texorpdfstring{$n$}{n}}
\flushleft{\textbf{Case $n=0$ - objects}}

A $0$-cube in $\AugOrdSet$ is just an object $X\in\AugOrdSet$. Therefore when applying the construction there are no maps and all we need to consider is a composition of $0$ maps, namely $\Id_X$. So $H_{comb}(X)=H_{comb}(\Id_X)$.

\begin{Example}
The first few values of $H_{comb}$ on the objects of $\AugOrdSet$ are as follows (given for clarity along with the imbedding in $\aug{X}$).
\begin{itemize}
\item $H_{comb}(\ord{0})=\ordop{0}$
\item $H_{comb}(\ord{1})=\ordop{1}$ 
\item $H_{comb}(\ord{2})$ is the horn 
\includegraphics[scale=0.3]{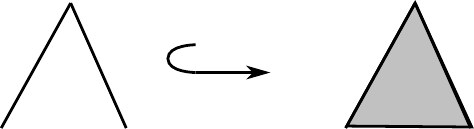}
\item $H_{comb}(\ord{3})$ is \includegraphics[scale=0.3]{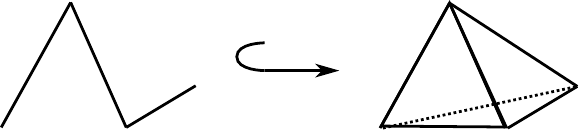}
\end{itemize}
\end{Example}

\flushleft{\textbf{Case $n=1$ - arrows}}

Let $X \xrightarrow{f} Y$ be a map in $\AugOrdSet$. We associate to it the correspondence 
\[
H_{comb}(X) \rightarrow H_{comb}(f) \leftarrow H_{comb}(Y)
\]

\begin{Example}
The Hall algebra multiplication comes from the image of the map $\ord{2} \to \ord{1}$, and on the level of $H_{comb}$ this map goes to 
\[
\includegraphics[scale=0.5]{Figures/2MultComb.pdf}
\]
\end{Example}


\flushleft{\textbf{Case $n=2$ - squares}}

Given a square in $\OrdSet$
\[
\stik{1}{
X \ar{r}{f} \ar{d}{g} \& Y \ar{d}{h}\\
Z \ar{r}{k} \& W
}
\]
we consider the map $X\xrightarrow{\alpha}W$ where $\alpha=h\circ f=k\circ g$ and we then construct the square of correspondences

\[
\stik{1}{
H_{comb}(X) \ar{r} \ar{d} \& H_{comb}(f) \ar{d}{}\& H_{comb}(Y) \ar{l} \ar{d} \\
H_{comb}(g) \ar{r} \& H_{comb}(\alpha) \& H_{comb}(h) \ar{l}{} \\
H_{comb}(Z) \ar{u} \ar{r} \& H_{comb}(k) \ar{u} \& H_{comb}(W) \ar{l} \ar{u}
}
\]

\begin{Example}
\label{Ex:AssociatorSquareHcomb}
The square\[
\stik{1}{
\ord{3} \ar{r}{p_1} \ar{d}{p_2} \& \ord{2} \ar{d} \\
\ord{2} \ar{r} \& \ord{1}
}
\]
with $p_1=(0\mapsto 0,1\mapsto 0,2\mapsto 1)$ and  $p_2=(0\mapsto 0,1\mapsto 1,2\mapsto 1)$ maps to
\[
\includegraphics[scale=0.8]{Figures/AssocInSset.pdf}
\]
\end{Example}

\subsubsection{Proof of \autoref{Prop:HCombCubesConstructionCommutes}}
\label{ProofOf:Prop:HCombCubesConstructionCommutes}
\begin{proof}
As $\sset$ is a regular category we only need to prove that all the squares in the cube commute where the maps are given by \autoref{Lem:GeneralCombConstruction}. The possible cases are covered by:

\flushleft{\textbf{Case 1:} Given maps}
\[
\stik{1}{
X\ar{r}{f}\& Y\ar{r}{g}\& Z\ar{r}{h} \& W
}
\]
we have a square of the form
\[
\stik{1}{
H_{comb}(g) \ar{r} \ar{d} \& H_{comb}(h\circ g) \ar{d}\\
H_{comb}(g\circ f) \ar{r} \& H_{comb}(h\circ g\circ f)
}
\]

 Moving horizontally essentially does nothing, and moving vertically is precomposing with $f$ in both cases, hence the square trivially commutes.
 
\flushleft{\textbf{Case 2:} Given a diagram of maps}
\[
\stik{1}{
S\ar{dr}{i}\\
{} \& X \ar{r}{f_1} \ar{d}{f_2} \& Y \ar{d}{g_1}\\
{} \& Z \ar{r}{g_2} \& W \ar{dr}{j}\\
{} \& {} \& {} \& T
}
\]
we have squares of the form
\[
\stik{1}{
H_{comb}(i) \ar{r} \ar{d} \& H_{comb}(f_1\circ i) \ar{d}\\
H_{comb}(f_2\circ i) \ar{r} \& H_{comb}(g_1\circ f_1\circ i = g_2\circ f_2\circ i)
}
\]
\[
\stik{1}{
H_{comb}(j) \ar{r} \ar{d} \& H_{comb}(j\circ g_1) \ar{d}\\
H_{comb}(j\circ g_2) \ar{r} \& H_{comb}(j\circ g_1\circ f_1 = j\circ g_2\circ f_2)
}
\]
whose commutativity is an immediate consequence of the commutativity of the diagram we started with.
\end{proof}

\subsection{Extension to \texorpdfstring{$\Spaces$}{Spc}}
\label{Hgeo}
Let $S_\bullet: \OrdSet^{op} \rightarrow \Spaces$ as in \autoref{def:SimplicialObjectInSpaces}. We note that it has a canonical extension to a functor $\sset^{op} \rightarrow \Spaces$:
\begin{Fact}
\label{S_ext}
The right Kan extension of $S_\bullet$ along the Yoneda embedding functor $\OrdSet^{op} \rightarrow \sset^{op}$ exists because the category $\Spaces$ is complete.
\end{Fact}
We will denote this extension also by $S_\bullet$. 

Now for every commutative cube in $\AugOrdSet$ we can associate a cube of correspondences in $\Spaces$ by composing $S_\bullet$ with $H_{comb}$. 

\begin{Proposition}
\label{CorrMonoidal}
$\Corr_k(S)\circ H_{comb}$ sends ordered disjoint unions of cubes in $\AugOrdSet$ to Cartesian products of cubes in $\Corr_k(\Spaces)$.
\end{Proposition}
\begin{proof}
Being a right Kan extension, $S$ sends limits in $\sset^{op}$ (i.e. colimits of simplicial sets) to limits in $\Spaces$.
It is clear from \autoref{Def:HCombOfGenerealMaps} that $H_{comb}$ sends the ordered disjoint union of ordered sets to a colimit of simplicial sets over a point, and then $S$ (being a Kan extension) sends it to a limit over $S$ applied to a point, which is the final object of $\Spaces$ by assumption. In all $\Corr_k(S)\circ H_{comb}:\AugOrdSet \rightarrow \Corr_k{\Spaces}$ sends products to products and we are finished.
\end{proof}

\begin{Remark}
The assumption that $S_0$ is the final object of $\Spaces$ forces us into the situation of an \emph{algebra} object, or a category with a single object. Without this assumption we arrive to the situation of an $\Ainfinity$ category object.
\end{Remark}

\section{\texorpdfstring{$\Ainfinity$}{A infinity}-algebras and higher Segal conditions}

\label{sec:HigherSegal}

$d$-Segal conditions were introduced in \cite[\S 2.3]{KapranovDyckerhoff}. For the definition and some basic results we will use various technical results about polytopes. A good source is \cite{zieglerPolytopes}.

Let us recall here the notion of a $d$-Segal object as outlined in \cite{KapranovDyckerhoff} and described in detail in \cite{PoguntkeSegal}. We start by recalling the $d=2$ case.

\subsection{2-Segal Conditions}
Let $S$ be a simplicial object in $\CCC$, i.e. a functor $S:\OrdSet^{op}\to\CCC$. We denote also by $S$ the right Kan extension of $S$ to $\sset^{op}$. In particular note that this extended $S$ takes colimits to limits.

Let $\mathcal{T}$ be a triangulation of an $n+1$-gon into triangles $T_1,\ldots,T_k$. e.g.
\[
\begingroup%
  \makeatletter%
  \providecommand\color[2][]{%
    \errmessage{(Inkscape) Color is used for the text in Inkscape, but the package 'color.sty' is not loaded}%
    \renewcommand\color[2][]{}%
  }%
  \providecommand\transparent[1]{%
    \errmessage{(Inkscape) Transparency is used (non-zero) for the text in Inkscape, but the package 'transparent.sty' is not loaded}%
    \renewcommand\transparent[1]{}%
  }%
  \providecommand\rotatebox[2]{#2}%
  \ifx\svgwidth\undefined%
    \setlength{\unitlength}{189.7157636bp}%
    \ifx\svgscale\undefined%
      \relax%
    \else%
      \setlength{\unitlength}{\unitlength * \real{\svgscale}}%
    \fi%
  \else%
    \setlength{\unitlength}{\svgwidth}%
  \fi%
  \global\let\svgwidth\undefined%
  \global\let\svgscale\undefined%
  \makeatother%
  \begin{picture}(1,0.79858246)%
    \put(0,0){\includegraphics[width=\unitlength,page=1]{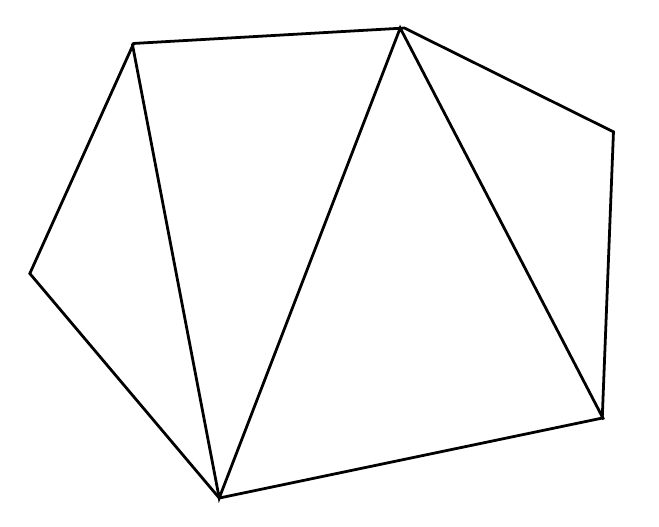}}%
    \put(0.17094233,0.74538743){\color[rgb]{0,0,0}\makebox(0,0)[lb]{\smash{$0$}}}%
    \put(0.60988878,0.76654437){\color[rgb]{0,0,0}\makebox(0,0)[lb]{\smash{$1$}}}%
    \put(0.92297273,0.6131971){\color[rgb]{0,0,0}\makebox(0,0)[lb]{\smash{$2$}}}%
    \put(0.30745386,0.00619754){\color[rgb]{0,0,0}\makebox(0,0)[lb]{\smash{$4$}}}%
    \put(-0.0035003,0.37465695){\color[rgb]{0,0,0}\makebox(0,0)[lb]{\smash{$5$}}}%
    \put(0.91871311,0.12972726){\color[rgb]{0,0,0}\makebox(0,0)[lb]{\smash{$3$}}}%
    \put(0.10065815,0.38118768){\color[rgb]{0,0,0}\makebox(0,0)[lb]{\smash{$T_1$}}}%
    \put(0.30953915,0.52774804){\color[rgb]{0,0,0}\makebox(0,0)[lb]{\smash{$T_2$}}}%
    \put(0.50122319,0.28920785){\color[rgb]{0,0,0}\makebox(0,0)[lb]{\smash{$T_3$}}}%
    \put(0.73550375,0.55756556){\color[rgb]{0,0,0}\makebox(0,0)[lb]{\smash{$T_4$}}}%
  \end{picture}%
\endgroup%

\]
By wrapping the polygon onto the $n$-simplex $\ordop{n}$ (matching up the vertices) this defines a map of simplicial sets \[
{T_1}\coprOver{{T_1\cap T_2}} {T_2}\coprOver{{T_2\cap T_3}} {T_3}\sqcup\ldots \coprOver{{T_{k-1}\cap T_k}} {T_k} \to \ordop{n}
\]
and by (contravariant) functoriality of $S$ defines a map
\[S_{\ordop{n}} \xrightarrow{\alpha_{{}_\mathcal{T}}} S_{T_1}\prodOver{S_{{}_{T_1\cap T_2}}} S_{T_2}\prodOver{S_{{}_{T_2\cap T_3}}} S_{T_3}\times\ldots \prodOver{S_{{}_{T_{k-1}\cap T_k}}} S_{T_k}\] 

\begin{Definition}
\label{2SegalAdaptednessDef}
We say that $S$ is \emph{adapted} to a polygonal triangulation $\mathcal{T}$ if the map $\alpha_{{}_\mathcal{T}}$ is an isomorphism, i.e. if it presents $S_{\ordop{n}}$ as the limit of the corresponding diagram.
\end{Definition}

\begin{Definition}
A functor $S:\OrdSet^{op}\to\CCC$ is said to be a 2-Segal object if $S$ is adapted to any polygonal triangulation $\mathcal{T}$ of an $n$-gon for $n\geq 4$.
\end{Definition}

\subsection{Higher Segal conditions}

Let $S$ be a simplicial object in $\CCC$, i.e. a functor $\OrdSet^{op}\to\CCC$, and let $d\geq 2$ be an integer.

Just as the 2-Segal conditions arise from triangulations of polygons, so the higher Segal conditions arise from higher dimensional triangulations of higher dimensional polytopes, as follows:

Let $C_d(n)$ be the $d$-dimensional cyclic polytope on $n+1$ vertices. It is the unique (up to isomorphism) $d$-dimensional polytope on $n+1$ vertices where no $d+1$ vertices are $d-1$-colinear. It will be convenient to give an explicit model, so consider the \emph{moment curve} in $\RR^d$: \[
\gamma_d:t\mapsto (t,t^2,\ldots,t^d)
\]
and for a set $S=\{s_1,\ldots,s_n\}\subset \RR$ define $C_d(S)$ to be the convex envelope of the points $\gamma_d(s_i)$. In particular let $C_d(n):=C_d(\{0,\ldots,n\})$.

An important feature of cyclic polytopes is the following:

\begin{Fact}
\label{Lem:CyclicPolytopes}
\begin{enumerate}
    \item We have projections $C_d(n)\xrightarrow{p_d^n} C_{d-1}(n)$ (in the above model the projections are given by omitting the last coordinate).
    \item For any map of ordered sets $I\xrightarrow{\varphi}J$ we have an induced map $C_d(I)\xrightarrow{\varphi_d} C_d(J)$.
    \item When $\varphi$ is injective the square
    \[
    \stik{1}{
    C_d(I)\ar{r}{\varphi_d} \ar{d}{p_d^n} \& C_d(J) \ar{d}{p_d^{n+1}}\\
    C_{d-1}(I)\ar{r}{\varphi_d}  \& C_{d-1}(J)
    }
    \]
    is a pullback.
\end{enumerate}
\end{Fact}

Via \autoref{Lem:CyclicPolytopes} we can see that covers of cyclic polytopes by other cyclic polytopes are directly related to systems of maps of ordered sets. We wish to single out two types of covers called the \emph{upper} and \emph{lower} triangulations.
For the purposes of our proofs we will only need a combinatorial description of these triangulations given below, but let us first remind of the geometric description:
\begin{Definition}
A point $x$ of the boundary of $C_d(n)$ is called \emph{upper} (resp. \emph{lower}) if $x+\RR_{>0}\cap C_d(n)=\emptyset$ (resp. $x-\RR_{>0}\cap C_d(n)=\emptyset$).
\end{Definition}

Note that the projection $C_d(n+1)\to C_d(n)$ sends the upper (resp. lower) part of the boundary onto $C_d(n)$ and therefore defines a triangulation of $C_d(n)$.

\begin{Definition}
The triangulation of $C_d(n)$ induced by the upper (resp. lower) part of the boundary of $C_d(n+1)$ is called the upper (resp. lower) triangulation of $C_d(n)$, and is denoted $\upperSeg{n,d}$ (resp. $\lowerSeg{n,d}$).
\end{Definition}

The following gives a purely combinatorial description of these triangulations due to \cite{galePolytopes}.

\begin{Definition}
An inclusion of ordered sets $I\subset J$ is called \emph{even} (resp. \emph{odd}) if for any $j\not\in I$ there are an even (resp. odd) number of elements $i\in I$ greater than $j$.
\end{Definition}

\begin{Proposition}[\cite{zieglerPolytopes} Theorem 0.7]
\label{Prop:loweruppersegal}
The even (resp. odd) subsets of size $d+1$ in $\{0,1,\ldots,n\}$ give the lower (resp. upper) triangulation of $C_d(n)$ by considering the corresponding imbeddings $\ordop{d}\cong C_d(d)\hookrightarrow C_d(n)$.
\end{Proposition}

As in the 2-Segal case, we can ask whether $S$ is adapted to these triangulations. Namely we can wrap $C_d(n)$ onto $\ordop{n}\cong C_n(n)$ and ask whether the natural maps from $S_{\ordop{n}}$ to the product induced by the triangulation is an isomorphism.

\begin{Definition}
\label{Def:HigherSegalObject}
A simplicial object $S$ is called
\begin{enumerate}
    \item  upper $d$-Segal if it is adapted to $\upperSeg{n,d}$ for all $n\geq d$.
    \item  lower $d$-Segal if it is adapted to $\lowerSeg{n,d}$ for all $n\geq d$.
    \item  (fully) $d$-Segal if it is both upper and lower $d$-Segal.
\end{enumerate}
\end{Definition}

This generalizes the $d=2$ case:

\begin{Proposition}[cf. \cite{PoguntkeSegal} Proposition 2.5]
\label{PropPog:SegalTriangulations}
$S$ is adapted to to $\lowerSeg{n,d}$ and $\upperSeg{n,d}$ iff $S$ is adapted to any triangulation of $C_d(n)$.
\end{Proposition}

We recall also the following:

\begin{Proposition}[\cite{PoguntkeSegal} Proposition 2.10]
\label{Prop:DSegalD+1Segal}
Let $S$ be a simplicial object in a category $\CCC$ which admits limits. Assume that $S$ is lower or upper $d$-Segal. Then S is fully $(d+1)$-Segal.
\end{Proposition}

\subsection{Higher Segal conditions and lax associativity}

\subsubsection{Main theorem}
In \autoref{sec:AssociativityCubes} we defined a system of cubes $A_n$ in $\AugOrdSet$ called the "associator cubes". Applying the construction of \autoref{sec:CombinatorialHall} to these cubes we obtain a system $H_n:=H_{comb}(A_n)$ of cubes of correspondences in $\sset^{op}$.

\begin{Theorem}
\label{thm:HigherSegalCubes}
Let $S$ be a simplicial object, i.e. a functor $\OrdSet^{op}\rightarrow \Spaces$, where $\Spaces$ is a complete $(\infty,1)$-category and assume that $S(\ordop{0})$ is the final object of $\Spaces$. Denote also by $S$ its right Kan extension along the functor $\OrdSet^{op}\rightarrow \sset^{op}$.
Let $d\geq 2$, then $S$ is $d$-Segal if and only if it sends every $H_n, n\geq d$ to an invertible cube in $\CORR_n(\Spaces)$.
\end{Theorem}

This implies:

\begin{Theorem}
\label{thm:HigherSegalLax}
    A $d$-Segal object in $\Spaces$ which sends $\ordop{0}$ to the final object of $\Spaces$ defines a $(d-1)$-lax $\Ainfinity$ algebra object in the $(\infty,d)$-category $\Corr_d(\Spaces)$.
\end{Theorem}

\begin{proof}
By \autoref{thm:HigherSegalCubes} $S$ is a $d$-Segal object iff $S(H_n)\in\CORR_{n,d}(\Spaces)$ for all $n$. By \autoref{Prop:CORRInvertibleCorr} $\HaugsengU(\CORR_{n,d}(\Spaces))\cong\Corr_d(\Spaces),\forall n\geq d$ and so we get a system of cubes in $\Corr_d(\Spaces)$. Since they are the images of the associator cubes under the composition $S\circ H_{comb}$ they obviously satisfy the requirements of \autoref{Def:AInfinityDatum}.
\end{proof}

\begin{Corollary}
A 2-Segal object $S$ in $\CCC$ defines a non-unital $\Ainfinity$ algebra in $\Corr_1(\CCC)$.
\end{Corollary}

\begin{proof}
From \autoref{thm:HigherSegalLax} we get that the images of all associator cubes are invertible, and so this gives us a $\AinfinityNU$ object in $\Corr_1(\CCC)$ in the sense of \autoref{Def:AInfinityDatum}. By \autoref{Prop:AinfinityDatumLurie} this is equivalent to giving a non-unital $\Ainfinity$ algebra structure on $S_1$.
\end{proof}

\subsubsection{Proof of \autoref{thm:HigherSegalCubes}}

We will need the following reformulation of the $d$-Segal consitions:
\begin{Proposition}
\label{Prop:DSegalBySteps}
Let $S$ be a simplicial object in $\CCC$. The following are equivalent
\begin{enumerate}
    \item $S$ is $d$-Segal.
    \item $S$ is adapted to $\upperSeg{n+1,n}$ and $\lowerSeg{n+1,n}$ for all $n\geq d$.
\end{enumerate}
\end{Proposition}

\begin{proof}
$1\Rightarrow 2$ follows immediately from \autoref{Prop:DSegalD+1Segal}.

For the converse, using induction on $n$, we need to show that if for any $k\leq n$ $S$ is adapted to $\lowerSeg{k,d},\upperSeg{k,d}$ and $S$ is adapted to $\lowerSeg{n+1,n},\upperSeg{n+1,n}$ then $S$ is adapted to $\lowerSeg{n+1,d},\upperSeg{n+1,d}$.

From \autoref{PropPog:SegalTriangulations}, the above condition is equivalent to $S$ being adapted to any triangulation of $C_d(n)$ or $C_n(n+1)$, and what we want to show is equivalent to $S$ being adapted to any triangulation of $C_d(n+1)$. So consider a triangulation $\TTT_d$ of $C_d(n+1)$, and a triangulation $\TTT_n$ of $C_n(n+1)$. By \autoref{Lem:CyclicPolytopes} $\TTT_n$ induces a cover of $C_d(n+1)$ by $C_d(n)_\alpha$'s. The triangulation $\TTT_d$ induces triangulations $\TTT_{d,\alpha}$ of the $C_d(n)_\alpha$'s in the cover, and since $S$ is adapted to them the limit over each one is $S_n$. This implies that the map $S_{n+1}\to S^{\TTT_d}$ factors as \[
S_{n+1}\to S_{\TTT_n}\to \lim_\alpha S_{\TTT_{d,\alpha}}\]
and it is clear that both of these maps are isomorphisms.
\end{proof}
Using the above we see that \autoref{thm:HigherSegalCubes} is equivalent to:
\begin{Proposition}
\label{Prop:HigherSegalReduction}
$S$ sends $H_d$ to an invertible cube $\mathcal{A}_d$ iff $S$ is adapted to $\lowerSeg{d+1,d}$ and $\upperSeg{d+1,d}$.
\end{Proposition}

\begin{Example}[\texorpdfstring{$d=2$}{d=2}]
\label{Ex:2AssCubeComb}
Consider the case $d=2$ which was discussed in \autoref{Ex:AssociatorSquareHcomb}. The associator $2$-cube is \[
\stik{1}{
\ord{3} \ar{r} \ar{d} \& \ord{2} \ar{d} \\
\ord{2} \ar{r} \& \ord{1}
}
\]
and so its image under $H_{comb}$ is $H_2$:
\[
\includegraphics[scale=0.8]{Figures/AssocInSset.pdf}
\]

In order for this square to go to an invertible square of correspondences in $\Spaces$, $S$ should take the squares in the upper right (i.e. $(1,0)$) and lower left (i.e. $(0,1)$) corners to pullback squares. One easily checks that this is equivalent to conditions $\lowerSeg{3,2}$ and $\upperSeg{3,2}$ respectively.
\end{Example}

\begin{Example}[\texorpdfstring{$d=3$}{d=3}] 
\label{Ex:3AssCubeComb}
The associator cube is
\begin{equation}
\label{PentagonCube}
\stik{1}{
 \& \ord{3} \ar{rr} \ar{dd} \& {} \& \ord{2} \ar{dd}\\
\ord{4} \ar[crossing over]{rr} \ar{dd} \ar{ur} \& {} \& \ord{3} \ar{dd} \ar{ur} \& {}\\
{} \& \ord{2} \ar{rr} \& {} \& \ord{1}\\
\ord{3} \ar{rr} \ar{ur} \& {} \& \ord{2} \ar{ur}
\latearrow{commutative diagrams/crossing over}{2-3}{4-3}{}
}
\end{equation}

Its image under $H_{comb}$ is a cube of correspondences of $\sset^{op}$, that is, a $2\times 2\times 2$ grid of commutative cubes in $\sset^{op}$ so that the outer shell is comprised of the images of the faces of the cube \autoref{PentagonCube} and the center is the 4-simplex. 

In order for this cube to be invertible by \autoref{invertiblecubes} the cubes in the upper-right-back (i.e. $(1,0,1)$) and lower-left-front (i.e. $(0,1,0)$) corners should go to pullback cubes under $S$. Let's consider first the upper-right-back cube:
\[
\includegraphics[scale=1.2]{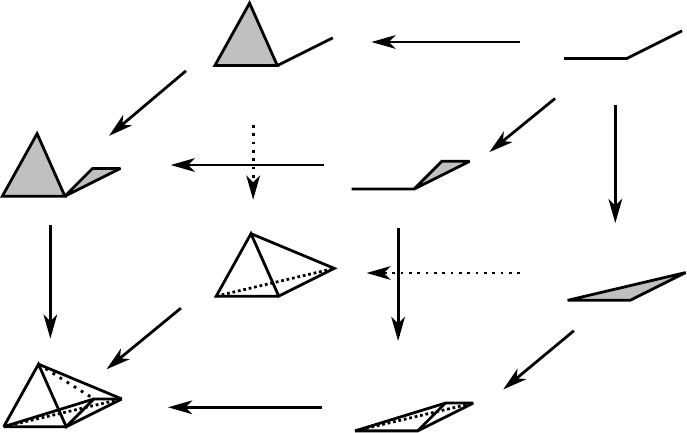}
\]

Its top face goes to a product of degenerate squares, hence a pullback. Therefore by \autoref{Cor:pullbackNcube} the cube is a pullback iff the bottom face is a pullback, and this is exactly $\upperSeg{4,3}$.

Now consider the lower-left-front cube:
\[
\includegraphics[scale=1.2]{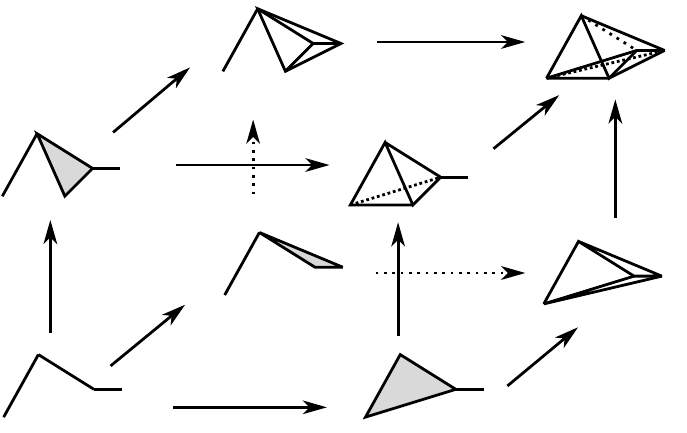}
\]
\end{Example}
The elements in the cube closest to $\ordop{4}$ correspond to the subsets $\{1,2,3,4\},\{0,1,3,4\},\{0,1,2,3\}$ and so the condition that this cube is a pullback is exactly $\lowerSeg{4,3}$.

Fix $d\geq 2$. As we explain below we can without loss of generality suppose that $d$ is odd. 
Let $u=(+,-,+,-,\ldots,+)$,$l=(-,+,-,+,\ldots,-)$. Recall from \autoref{sec:CorrCubes} that a cube of correspondences has vertices indexed by the faces of the lax cube, and that these are indexed by sequences in the symbols $\{-,0,+\}$ as outlined in \autoref{sec:LaxCubes}. Using \autoref{not:corrsubcube} for the subcubes of a cube of correspondences in $\CCC$ we must prove the following reformulation of \autoref{Prop:HigherSegalReduction}:

\newpage

\begin{Proposition}
\label{thm:lowerupperSegal}
\begin{itemize}
    \item[]
    \item $S$ is adapted to $\lowerSeg{d+1,d}$ iff $S((H_d)^{00\ldots0}_l)$ is a pullback cube.
    \item $S$ is adapted to $\upperSeg{d+1,d}$ iff $S((H_d)^{00\ldots0}_u)$ is a pullback cube.
\end{itemize}
\end{Proposition}
\begin{proof}
Let $v$ be a sequence in $\{-,0,+\}$ and denote by $H_d(v)$ the simplicial set in the $v$ position in the cube of correspondences $H_d$. Denote $s:=00\ldots 0$.
Recall from \autoref{sec:CombinatorialHall} that $H_d(s)=H_{comb}(\ord{d+1}\to\ord{1})$. 

Let $v_i^+$ (resp. $v_i^-$) be the sequence with $0$ everywhere except at the $i^\text{th}$ place where it has a $+$ (resp. $-$). 
Let us compute the maps $H_d(v_i^\pm)\to H_d(s)$.

Consider first the $v_i^+$ case. $v_i^+$ corresponds to the codimension $1$ face of the associator cube which is generated by all paths starting from the vertex one gets by travelling from $\ord{d+1}$ in the $i^\text{th}$ direction, and ends in $\ord{1}$ (cf. \autoref{def:Hcomb}). This means that the map $H_d(v_i^+)\to H_d(s)$ is the map \[
H_{comb}(\ord{d}\to\ord{1})\to H_{comb}(\ord{d}\to\ord{1}\circ p_i)=H_{comb}(\ord{d+1}\to\ord{1})=\ordop{d+1}
\]
where $p_i:\ord{d+1}\to\ord{d}$ is the surjection with $p_i(i)=p_i(i+1)=i$.

Therefore from \autoref{def:Hcomb} this is the imbedding of the $d$-simplex into the $d+1$-simplex corresponding to the subset inclusion $\{0,1,\ldots \hat{i},\ldots,d+1\}\subset \{0,1,\ldots,d+1\}$.

Now consider the $v_i^-$ case. $v_i^-$ corresponds to the face of the associator cube starting at $\ord{d+1}$ and ending at $\ord{2}$ which is generated by moving in all directions \emph{except} the $i^\text{th}$. Therefore $H_d(v_i^-)$ is $H_{comb}$ applied to the map $\ord{d+1}\xrightarrow{q_i}\ord{2}$ corresponding to the partition $\ord{d+1}=1,2,\ldots,i\sqcup i+1,\ldots,d+1$. The map $H_d(v_i^-)\to H_d(s)$ is then the map
\[
H_{comb}(q_i)\to H_{comb}((\ord{2}\to\ord{1})\circ q_i)=H_{comb}(\ord{d+1}\to\ord{1})=\ordop{d+1}
\]
In terms of simplicial sets this is the imbedding of the gluing of two lower dimensional simplices over a point into a $d+1$-simplex corresponding to the decomposition of sets $\{0,1,2,\ldots,i\}\sqcup_{\{i\}} \{i,\ldots,d+1\}$.

Let us now prove the $\lowerSeg{d+1,d}$ case. Recall from \autoref{Prop:loweruppersegal} that the triangulation $\lowerSeg{d+1,d}$ is determined by all even subsets of size $d+1$ in $\{0,1,\ldots,d+1\}$. The target vertex of the cube $(H_d)^{00\ldots0}_l$ is $H_d(00\ldots0)$. The vertex distance 1 from the target in the $i^{\text{th}}, 1\leq i \leq d$ direction inside $(H_d)^{00\ldots0}_l$ is $H_d(v_i^{\sgn{(-1)^i}})$. In particular, if $i$ is even the map to the target corresponds to the inclusion of the size $d+1$ even set $\{0,1,\ldots \hat{i},\ldots,d+1\}$. Additionally, if $i=1$ the map corresponds to the inclusion of the union of the size $d+1$ even set $\{1,\ldots,d+1\}$ with the set $\{0,1\}$ and similarly for $i=d$. In all we get all the simplexes corresponding to even subsets of size $d+1$, along with some extra terms in directions $1,d$.

Therefore we see that $S$ is adapted to $\lowerSeg{d+1,d}$ iff $S$ applied to the subcube generated by the directions $1,d$ and all even directions is a pullback cube (it is straightforward to check that the extra terms in the $1$ and $d$ directions cancel out in the pullback diagram as in \autoref{Ex:2AssCubeComb} and \autoref{Ex:3AssCubeComb}). Denote this subcube $(H_d)_\text{lower}$.

To see that this is equivalent to the image of the whole cube being a pullback consider an odd direction $j$ different from $1,d$ and consider the cube generated by the above and this extra direction.
It has a codimension 1 face $(H_d)_\text{lower}$.
We want to examine the opposite face. Denote it by $C$.
Using the same considerations as above we see that the target vertex of $C$ is $H_{comb}(v_j^-)$ and so it is equal to the union $\ordop{j}\coprOver{\ordop{0}}{\ordop{d-j+1}}$.
Hence from \autoref{def:Hcomb} it follows that $C$ decomposes into a union of cubes $C_1\coprOver{\ordop{0}} C_2$ in a compatible way.

Moreover, we claim that for each fixed direction in the face in question the maps in either $C_1$ or $C_2$ are identity maps (as in \autoref{Ex:3AssCubeComb}).

To see this note that in terms of the indexing via cube faces $C=S((H_d)^{v_j^-}_{l'})$ for $l':=(-,+,0,+,0,+,0,\ldots, -)$, i.e. $l'$ is a sequence with $0$'s appearing in all odd places except $1,j,d$. In other words $C$ is the cube generated by paths going from $v_j^-$ in all the directions that appear in the cube $(H_d)_{lower}$. When moving in each direction we are either changing only in the first part (before the "$-$") i.e. in $C_1$ or only in the second part i.e. in $C_2$ and keeping the other part fixed.

Now since in both $C_1$ and $C_2$ there are non-trivial maps (since both $(d+1)-j$ and $j$ are more than $1$) this implies that they are both degenerate cubes (as in \autoref{Def:degenerateCube}). Applying $S$ we obtain a product of two degenerate cubes which is therefore automatically a pullback cube.

Using \autoref{Cor:pullbackNcube} we see that the image of the whole subcube generated by $(H_d)_\text{lower}$ and direction $j$ is a pullback cube iff the image of $(H_d)_\text{lower}$ is. Proceeding in the same way by induction we get that $S((H_d)^{00\ldots0}_l)$ is a pullback cube iff $S((H_d)_\text{lower})$ is a pullback cube iff $S$ is adapted to $\lowerSeg{d+1,d}$.

The proof for the $\upperSeg{d+1,d}$ case, and the $d$ even cases is exactly the same, replacing even with odd and making the obvious adjustments.
\end{proof}

\begin{appendices}
\section{Segal model for \texorpdfstring{$(\infty,n)$}{(infinity,n)}-categories}

\label{sec:uppleSegal}
Let us denote by $\CAT{1}$ the $(\infty,1)$-category of infinity categories, and let $\Spc \hookrightarrow \CAT{1}$ be the $(\infty,1)$-category of $\infty$-groupoids. The realization of $\CAT{1}$ by quasicategories was introduced by Joyal \cite{Joyal2008} and later extensively developed by Lurie in \cite{LurieTopos, LurieAlgebra}. The realization via \emph{complete Segal spaces} was proposed by Rezk in \cite{RezkSegal}. The latter allows for a certain generalization which provides a model for $(\infty,n)$-categories in  \cite{BarwickSegal}. We briefly recall the relevant background in this section.

\subsection{Segal objects}

\begin{Definition}
A \emph{Segal space} is a simplicial space $X:\OrdSet^{op}\to\Spc$ that satisfies the following condition called the \emph{Segal condition} (it is equivalent to the special case $d=1$ of the $d$-Segal conditions): Decomposing $n=m+k$ and denoting by $X_n$ the image of $\ordop{n}$ we have a square in $\OrdSet$\[
\stik{1}{
\ordop{n} \& \ar{l} \ordop{m}\\
\ordop{k} \ar{u}\& \ar{l} \ar{u} \ordop{0}
}
\]
where $\ordop{m}\hookrightarrow\ordop{n}$ is the imbedding of the first $m+1$ elements and $\ordop{k}\hookrightarrow\ordop{n}$ is the imbedding of the last $k+1$ elements. The image of this square under $X$ is a square in $\Spaces$\[
\stik{1}{
X_n \ar{d} \ar{r} \& X_m\ar{d}\\
X_k \ar{r} \& X_0
}
\]
and we require this square to be a pullback.
\end{Definition}

The space $X_0$ should be thought of as the space of objects and the space $X_1$ as the space of arrows. The basic case of the Segal condition is the map $X_2\to X_1\times_{X_0} X_1$. Choosing an inverse for this map and composing it with the third map $X_2\to X_1$ gives a "composition of arrows" map for our category $X$. The spaces $X_n$ encode the higher associativity data for the composition of morphisms. 

The structure of a space on $X_1$ remembers the structure of isomorphisms of arrows in our category. 

Formally, the category of $(\infty,1)$-categories is given by the above construction after specifying a certain model structure and localizing its weak equivalences. The local objects in this model structure are \emph{complete} Segal spaces. The details can be found in \cite{RezkSegal}.
Note that the notion of Segal \emph{space} can be adapted verbatim to any target category with finite limits to yield the general notion of a Segal \emph{object} in that category. As in the $(\infty,1)$-case in order to obtain a model for $(\infty,n)$-categories we need to localize by a certain class of morphisms. A localization modelled on the notion of completeness for Segal spaces the was proposed in \cite{BarwickSegal} and leads to the following inductive definition of $(\infty, n)$-categories:

\begin{Definition}
An $n$-fold Segal space $X$ is a Segal object in the category of $(n-1)$-fold Segal objects such that the $n-1$-Segal object $X_0{,\ldots}$ is constant.
\end{Definition}

\begin{Definition}
The category of $(\infty,n)$-categories $\CAT{N}$ is the category of \emph{complete n-fold Segal objects} $X$ in $\CAT{(N-1)}$.
\end{Definition}

\begin{Example}[case n=2]
Unravelling the definition we get the following - a 2-fold Segal space can be identified with a bisimplicial space $X:\OrdSet^{\times 2}\to\Spaces$ such that \begin{itemize}
    \item The simplicial space $X_{0,-}$ is constant.
    \item The simplicial space $X_{i,-}$ is Segal.
\end{itemize}
\end{Example}

This tells us that if $X$ is the nerve of a category, then for instance $X_{1,1}$ consists of cells of the form 
\[
\stik{1}{
A \ar{r} \ar{d}[above,sloped]{\sim}\& B \ar{d}[above,sloped]{\sim} \ar[Rightarrow, shorten <=1em,shorten >=1em]{dl}\\
A' \ar{r}\& B'
}
\]
which are called \emph{Rezk type} cells.

\subsection{Underlying category functors}
\label{Seq:UnderlyingCategoryFunctor}
The canonical embedding $i:\Spc\to\CAT{1}$ has a right adjoint which we will denote $\square^{\Spc}$ and call the \emph{underlying space functor}.

By induction, this gives for any $N$ a functor \[\square^{\CAT{(N-1)}}:\CAT{N}\to\CAT{(N-1)}\] which we will call the \emph{underlying category functor}.

For example if $\CCC\in\CAT{2}$, then $\CCC^{\CAT{1}}$ is the simplicial object in spaces which sends $\ordop{n}$ to the space $\Maps(\ordop{n},\CCC)$.

\subsection{\texorpdfstring{$N$}{N}-uple Segal spaces}
In many cases, and in particular for our constructions in \autoref{sec:CombinatorialHall} a more cubical model is needed, i.e. we want to have more general cells of the form
\[
\stik{1}{
A \ar{r} \ar{d}[above,sloped]{}\& B \ar{d}[above,sloped]{} \ar[Rightarrow, shorten <=1em,shorten >=1em]{dl}\\
A' \ar{r}\& B'
}
\]
where we don't assume that two of the sides are equivalences. It is easy to see that the existence of such cells would interfere with the constancy condition in the definition of $N$-fold Segal spaces, and so (as in e.g. \cite{HaugsengSpans}) we introduce the auxilliary category of $N$-uple Segal spaces defined inductively as follows:

\begin{Definition}
 An $N$-uple Segal space is a Segal object in $(N-1)$-uple Segal spaces.
\end{Definition}

Unravelling the definition we can rephrase the definition as saying that an $N$-uple Segal space is a $\OrdSet^{\times N}$ space satisfying Segal conditions when we fix all indices but one. By definition, an $N$-fold Segal space is also an $N$-uple Segal space, and we have:
 
\begin{Proposition}[\cite{HaugsengSpans} \S4]
The inclusion of $N$-fold Segal spaces in $N$-uple Segal spaces has a right adjoint, denoted $\HaugsengU$.
\end{Proposition}

In terms of the types of cells that appear, the functor $\HaugsengU$ simply picks out the subcategory of Rezk type cells.

To see heuristically why this procedure does not lose significant information we can consider the following diagram:
\begin{equation}
    \label{Eq:SquareConnection}
    \stik{1}{
    A \ar[equals]{r} \ar[equals]{d} \& A  \ar[equals, shorten <=1em,shorten >=1em]{dl} \ar{r} \ar{d}[above,sloped]{}\& B \ar{d}[above,sloped]{} \ar[Rightarrow, shorten <=1em,shorten >=1em]{dl} \ar{r} \& B' \ar[equals]{d}  \ar[equals, shorten <=1em,shorten >=1em]{dl}\\
    A \ar{r} \& A' \ar{r}\& B' \ar[equals]{r}\& B'
    }
\end{equation}
Composing this diagram gives us a Rezk type cell, and in all the examples we will consider, there will always be a canonical way to construct the "trivial" left and right squares.
\section{Pullback cube criterion}

\label{sec:pullbacklemma}
\begin{Definition}
\label{Def:PBCube}
A commutative cube is said to be a \emph{pullback} cube if it presents the source vertex as the limit of the rest of the diagram. 
\end{Definition}

\begin{Lemma}
\label{Lem:pullbackcube}
Consider a cube in an $\infty$-category
\[
\stik{1}{
 \& A \ar{rr} \ar{dd} \& {} \& B \ar{dd}\\
X \ar[crossing over]{rr} \ar{dd} \ar{ur} \& {} \& Y \ar{dd} \ar{ur} \& {}\\
{} \& C \ar{rr} \& {} \& D\\
Z \ar{rr} \ar{ur} \& {} \& W \ar{ur}
\latearrow{commutative diagrams/crossing over}{2-3}{4-3}{}
}
\]
And suppose that $\squCorns{A}{B}{C}{D}$ is a pullback square, then $\squCorns{X}{Y}{Z}{W}$ is a pullback square if and only if the whole cube is a pullback cube.
\end{Lemma}

\begin{proof}
Before presenting the general proof it is instructive to consider the case of a usual 1-category:

Assume $\squCorns{X}{Y}{Z}{W}$ is a pullback square.

Consider another cube 
\[
\stik{1}{
 \& A \ar{rr} \ar{dd} \& {} \& B \ar{dd}\\
\widetilde{X} \ar[crossing over]{rr} \ar{dd} \ar{ur} \& {} \& Y \ar{dd} \ar{ur} \& {}\\
{} \& C \ar{rr} \& {} \& D\\
Z \ar{rr} \ar{ur} \& {} \& W \ar{ur}
\latearrow{commutative diagrams/crossing over}{2-3}{4-3}{}
}
\]
We want to show that there is a unique map $\widetilde{X}\to X$ that makes everything commute.

Since $XYZW$ is a pullback square we have a unique map $\widetilde{X}\to X$ such that $\widetilde{X}Y=XY\circ\widetilde{X}X$ and $\widetilde{X}Z=XZ\circ\widetilde{X}X$.

we just need to show that $\widetilde{X}A=XA\circ\widetilde{X}X$. This follows because both sides are a map $\widetilde{X} \to A$ which make the diagram
\[
\stik{1}{
\widetilde{X} \ar{dr} \ar{drr} \ar{ddr} \\
{} \& A \ar{d} \ar{r}\& B \ar{d} \\
{} \& C \ar{r} \& D
}
\]
commute. Since we assumed $ABCD$ is a pullback such a map is unique.

Assume now that the cube is a pullback and consider a square \[
\stik{1}{
\widetilde{X} \ar{r} \ar{d} \& Y \ar{d} \\
Z \ar{r} \& W
}
\]
We want to show that there is a unique map $\widetilde{X}\to X$ which makes the diagram 
\[
\stik{1}{
\widetilde{X} \ar{dr} \ar{drr} \ar{ddr} \\
{} \& X \ar{d} \ar{r}\& Y \ar{d} \\
{} \& Z \ar{r} \& W
}
\]
commute.

The compositions $YB\circ\widetilde{X}Y$ and $ZC\circ\widetilde{X}Z$ fit in a commutative square 
\[
\stik{1}{
\widetilde{X} \ar{r} \ar{d} \& B \ar{d} \\
C \ar{r} \& D
}
\]
and so there is a map $\widetilde{X}\to A$ that makes everything commute, and since the cube is a pullback this gives us our desired map $\widetilde{X}\to X$.

\flushleft{\textbf{General case:}}

Le us reformulate this proof in  $\infty$-categorical language. Using the standard approach, consider the map 
\[
\stik{1}{
{} \&  \bullet\ar{d} \\
 \bullet\ar{r} \& \bullet
}
\xhookrightarrow{}
\pt
\]
For any category $\Spaces$ this induces a functor $\Spaces\to\Spaces^\pbcorner$.

Given a corner, i.e. a map $\point\to\Spaces^\pbcorner$, we can consider the pullback in the ($\infty$-)category of ($\infty$-)categories:

\[
\stik{1}{
X_\pbsquare \ar{d} \ar{r}[above]{} \& \Spaces  \ar{d}{}\\
\point \ar{r}{} \&  \Spaces^\pbcorner
}
\]
This is just the category of cones over the corner and so a pullback of the corner is the same as a final object in $X_\pbsquare$.

We can do the same for 
\[
\pbcubecorner := \stik{0.3}{
 \& \bullet \ar{rr} \ar{dd} \& {} \& \bullet \ar{dd}\\
{}  \& {} \& \bullet \ar{dd} \ar{ur} \& {}\\
{} \& \bullet \ar{rr} \& {} \& \bullet\\
\bullet \ar{rr} \ar{ur} \& {} \& \bullet \ar{ur}
\latearrow{commutative diagrams/crossing over}{2-3}{4-3}{}
}
\]
to define a pullback cube and we denote the resulting pullback category $X_\pbcube$. A pullback of the given cube corner is then a final object of $X_\pbcube$. 

Note that a cube corner defines us two square corners - the back and front. We denote the pullback categories corresponding to them by $X_{\pbsquare-front}$ and $X_{\pbsquare-back}$. 

In analogy with the proof for 1 categories, we will prove:
\begin{Claim}
The restriction map $X_\pbcube\xrightarrow{res}X_{\pbsquare-front}$ is an equivalence.
\end{Claim}

By our assumption about the back face of the cube, it defines a final object $f_{back}\in X_{\pbsquare-back}$
We have the following commutative diagram of maps:\[
\stik{1}{
X_{\pbcube} \ar{r}{F} \ar{d}{res}\& {X_{\pbsquare-back}}_{/f_{back}} \ar{d}{G}\\
X_{\pbsquare-front} \ar{r}{H} \& X_{\pbsquare-back}
}
\]
where:
\begin{itemize}
    \item $F$ comes from the identification of a cube with a map of squares.
    \item $G$ is the source map.
    \item $H$ is given by composition with the (fixed) cube corner.
\end{itemize}
First note that this square is obviously a pullback square. Now since $f_{back}$ is a final object, the map $G$ is an equivalence, and so $res$ is as well.

To conclude the proof of \autoref{Lem:pullbackcube}: A cube with the given corner is a pullback iff it gives a final object in $X_\pbcube$ iff its front face gives a final object in $X_{\pbsquare-front}$, iff it is a pullback of the front corner.
\end{proof}

\begin{Corollary}
\label{Cor:pullbackNcube}
Let $C$ be an $n$-cube. Suppose that an $n-1$-subcube $C'$ in $C$ is a pullback cube, then $C$ is a pullback cube iff the opposite cube to $C'$ is a pullback cube.
\end{Corollary}

\begin{proof}
Proven in the same way as \autoref{Lem:pullbackcube} by induction on $n$.
\end{proof}
\end{appendices}

\printbibliography
\end{document}